\documentclass{amsart}
\usepackage{amssymb}
\usepackage{amsrefs}

\newtheorem{theorem}{Theorem}

\newtheorem{lemma}[theorem]{Lemma}

\newtheorem{prop}[theorem]{Proposition}

\theoremstyle{definition}
\newtheorem{definition}[theorem]{Definition}
\newtheorem{remark}[theorem]{Remark}

\newcommand{\lab}{\left|}
\newcommand{\rab}{\right|}
\newcommand{\bdry}{\partial_\infty}
\newcommand{\lp}{\left(}

\newcommand{\rp}{\right)}
\newcommand{\e}{\epsilon}
\newcommand{\R}{\mathbb{R}}
\newcommand{\C}{\mathbb{C}}

\newcommand{\Hy}{\mathbb{H}}

\newcommand{\Sp}{\mathbb{S}}

\newcommand{\acts}{\curvearrowright}
\newcommand{\Z}{\mathbb{Z}}

\newcommand{\N}{\mathbb{N}}

\newcommand{\ti}{\textit}

\DeclareMathOperator{\dist}{\textup{\text{dist}}}
\DeclareMathOperator{\diam}{\textup{\text{diam}}}

\DeclareMathOperator{\dimup}{\textup{\text{dim}}}
\DeclareMathOperator{\CAT}{\textup{\text{CAT}}}
\DeclareMathOperator{\AC}{\textup{\text{AC}}}
\DeclareMathOperator{\Vol}{\textup{\text{Vol}}}
\DeclareMathOperator{\topo}{\textup{\text{top}}}
\DeclareMathOperator{\vol}{\textup{\text{vol}}}
\DeclareMathOperator{\supp}{\textup{\text{supp}}}
\DeclareMathOperator{\inter}{\textup{\text{int}}}
\DeclareMathOperator{\GL}{\textup{\text{GL}}}

\numberwithin{equation}{section}
\numberwithin{theorem}{section}

\begin{document}

\title[Quasi-M\"obius actions on fractal spaces]{Rigidity for quasi-M\"obius actions on fractal metric spaces}
\author{Kyle Kinneberg}

\begin{abstract}
In \cite{BK02}, M. Bonk and B. Kleiner proved a rigidity theorem for expanding quasi-M\"obius group actions on Ahlfors $n$-regular metric spaces with topological dimension $n$. This led naturally to a rigidity result for quasi-convex geometric actions on $\CAT(-1)$-spaces that can be seen as a metric analog to the ``entropy rigidity" theorems of U. Hamenst\"adt \cite{Ham90} and M. Bourdon \cite{Bour96IHES}. Building on the ideas developed in \cite{BK02}, we establish a rigidity theorem for certain expanding quasi-M\"obius group actions on spaces with different metric and topological dimensions. This is motivated by a corresponding entropy rigidity result in the coarse geometric setting.
\end{abstract}

\date{\today}

\maketitle

\section{Introduction} \label{intro}

Let $(M^n,g)$ be a compact Riemannian manifold of dimension $n$. If we assume that $M$ is locally symmetric and negatively curved, then its universal cover is isometric to $\Hy^k_F$---one of the hyperbolic spaces defined over $F = \R, \C$, the quaternions, or the octonians (in the last case, only for $k=2$, which corresponds to real dimension $n=16$). We can therefore identify $(M^n,g)$ with the quotient $\Hy^k_F/\Gamma$, where $\Gamma = \pi_1(M)$ acts on $\Hy^k_F$ by deck transformations. A natural question then arises: Is this hyperbolic structure uniquely determined by the topology of $M$?

G. Mostow's classic rigidity theorem \cite{Most73} gives an affirmative answer to this question in dimensions $n \geq 3$. More specifically, he proves that if two locally symmetric compact manifolds, both with maximal sectional curvature $-1$, have isomorphic fundamental groups, then they are isometric. The curvature assumption here is simply a scaling normalization. In other words, the topology of a locally symmetric compact manifold determines its metric structure, up to scaling. 

It is important to note, of course, that there is no analogous theorem for surfaces. Indeed, a compact surface of genus $g \geq 2$ has a $(6g-6)$-dimensional moduli space of hyperbolic metrics (i.e., locally symmetric metrics of constant curvature $-1$). Such surfaces therefore have many metric deformations that would be ruled out by a rigidity theorem.

\subsection{Extending Mostow rigidity}

Let us now turn our attention to a different question: How can one determine when a negatively curved manifold is locally symmetric? A significant amount of work in this direction, much of it from the early and mid-1990s, sought to find a symmetric structure in manifolds that were extremal for certain metric quantities---volume, curvature bounds, geodesic lengths, entropy, etc. Most relevant for us here is the entropy rigidity theorem of U. Hamenst\"adt \cite{Ham90}.

Before stating this result, we must first define entropy. Let $(M^n,g)$ be a compact Riemannian manifold and let $(\tilde{M},\tilde{g})$ be its universal Riemannian cover with metric $\tilde{g}$. Let $B(p,R)$ denote the ball of radius $R$ centered at $p \in \tilde{M}$, and let $\Vol_{\tilde{g}}B(p,R)$ be the volume of this ball. We call
$$h_{\vol}(g) = \lim_{R\rightarrow \infty} \frac{\log \lp \Vol_{\tilde{g}} B(p,R) \rp}{R}$$
the \textit{volume entropy} of $g$; this limit is independent of the choice $p \in \tilde{M}$. For example, if $(M^n,g)$ is hyperbolic, then $(\tilde{M},\tilde{g})$ can be identified with real hyperbolic space $\Hy^n$. Consequently, 
$$\Vol_{\tilde{g}} B(p,R) = \Vol_{\Hy^n} B(p,R) \approx e^{(n-1)R},$$
so that $h(g) = n-1$.

The following relationship indicates why this volume-growth quantity is considered to be a type of entropy. Let $h_{\topo}(g)$ denote the \textit{topological entropy} of the geodesic flow on the unit tangent bundle of $(M^n,g)$ (see \cite[Section 3]{Mann79} for definitions). For general compact manifolds, Manning \cite{Mann79} showed that $$h_{\topo}(g) \geq h_{\vol}(g),$$ and if $(M^n,g)$ has non-positive sectional curvature, then equality holds. As we will concern ourselves only with compact manifolds of negative sectional curvature, we can set
$$h(g) = h_{\topo}(g) = h_{\vol}(g)$$
from now on and refer to it simply as the \textit{entropy of $g$}.

\begin{theorem}[Hamenst\"adt \cite{Ham90}] \label{Ham}
Let $(M^n,g_0)$ be a locally symmetric compact manifold with maximal sectional curvature $-1$ and $n\geq 3$. Let $g$ be another Riemannian metric on $M$, also with maximal sectional curvature $-1$. Then $h(g) \geq h(g_0)$, and equality holds if and only if $g$ is locally symmetric. In particular, equality holds if and only if $(M^n,g)$ is isometric to $(M^n,g_0)$.
\end{theorem}

In other words, the locally symmetric structures on $M$ are precisely the minima of the entropy functional, at least among metrics suitably normalized by curvature. Note also that the ``in particular" statement in this theorem follows from Mostow rigidity.

From Hamenst\"adt's theorem, there are various directions in which one may proceed (see, for example, the survey \cite{Spa04} on rigidity theory). Remaining in the Riemannian setting, we can ask if there are other pairs of normalizations and metric quantities for which rigidity theorems can be obtained. For example, suppose that $(M,g_0)$ is a locally symmetric compact manifold of dimension $\geq 3$ with unit volume and let $g$ be another metric on $M$ with unit volume. Then $h(g) \geq h(g_0)$ and equality holds precisely when $(M,g)$ and $(M,g_0)$ are isometric. A more general version of this was established by G. Besson, G. Courtois, and S. Gallot \cite{BCG96}, along with several consequential rigidity statements, but the general theme is that the locally symmetric metrics on manifolds related to $(M,g_0)$ can be identified by two quantities: volume and entropy. Incidentally, the methods used in their paper give a constructive proof of Mostow's original result by exhibiting the desired isometry.

\subsection{Toward a metric setting}

A different direction one may take (and the direction we wish to push further in this paper) is to extend Hamenst\"adt's theorem to metric geometry. To motivate the comparison between the Riemannian and metric settings, let $(M^n,g)$ be as in Theorem \ref{Ham}, and let $\Gamma = \pi_1(M)$ be its fundamental group. Also, let $(X,d)$ be its Riemannian universal cover with metric $d$. Of course $d$ is a Riemannian metric itself, but as we move away from the Riemannian setting, we want to think of $d$ simply as a distance function.

The negative curvature in $(M^n,g)$, which guarantees negative curvature in $(X,d)$ as well, allows one to define an ideal boundary: the collection of asymptotic classes of geodesic rays emanating from a fixed base-point. Moreover, this boundary has a canonical metric structure that is closely related to the asymptotic geometry of $X$. For example, if $g$ is hyperbolic, then its universal cover is, once again, the real hyperbolic space $\Hy^n$, whose ideal boundary is the Euclidean sphere $\Sp^{n-1}$. 

The isometric action of $\Gamma$ on $(X,d)$ passes naturally to an action on the ideal boundary. Here the entropy of $g$ plays an important role, as $h(g)$ is the Hausdorff dimension of the canonical metric on the boundary. Equality of $h(g)$ and $h(g_0)$ therefore guarantees that the boundary associated to $g$ has metric properties similar to those of the boundary associated to $g_0$, which is much better understood.

Let us make the comparison between Riemannian and metric geometry more explicit. The universal cover $(X,d)$ has sectional curvature at most $-1$, so it satisfies the $\CAT(-1)$ condition \cite[Th\'eor\`eme 9]{Tro90}. Recall that a geodesic metric space is called $\CAT(-1)$ if its geodesic triangles are thinner than their comparison triangles in the real hyperbolic plane. Moreover, the action of the fundamental group $\Gamma$ on $X$ is isometric, properly discontinuous, and cocompact. We will refer to such actions as \textit{geometric} actions from now on. Actually, to deal with more general situations, it will be convenient to weaken the cocompactness property to \textit{quasi-convex cocompactness}: there is a quasi-convex subset $Y \subset X$ on which $\Gamma$ acts cocompactly. We call such actions \ti{quasi-convex geometric}; see Section \ref{rcg} for formal definitions.

For a $\CAT(-1)$-space $X$, one can define a boundary at infinity, which we denote by $\bdry X$. As with the ideal boundary, it will be a topological space with a canonical metric structure (again, see Section \ref{rcg} for details). Let $\Lambda(\Gamma)$ be the limit set of $\Gamma$ in $\bdry X$. If the action is cocompact, then $\Lambda(\Gamma) = \bdry X$, but in general the limit set can be much smaller than the whole boundary. However, its Hausdorff dimension has a familiar form \cite[Th\'eor\`eme 2.7.4]{Bour96}:
$$\dimup_H \Lambda(\Gamma) = \limsup_{R \rightarrow \infty} \frac{\log(N(R))}{R},$$
where $p \in X$ is any point and $N(R) = \# \{ \Gamma p \cap B_X(p,R) \}$ is the number of points in the orbit $\Gamma p$ that lie at distance at most $R$ from $p$. This Hausdorff dimension is therefore a metric analog of the entropy we considered earlier.

In this context, M. Bourdon \cite{Bour96IHES} proved the following generalization of Theorem \ref{Ham}.

\begin{theorem}[Bourdon \cite{Bour96IHES}]
Let $\Gamma = \pi_1(M^n,g_0)$ be the fundamental group of a locally symmetric compact manifold of maximal sectional curvature $-1$ and dimension $n \geq 3$. Suppose that $\Gamma$ acts quasi-convex geometrically on a $\CAT(-1)$-space $X$. Let $S$ be the universal Riemannian cover of $(M^n,g_0)$. Then
$$\dimup_H \Lambda(\Gamma) \geq \dimup_H \bdry S,$$
and equality holds if and only if there is an isometric embedding $F \colon S \rightarrow X$, equivariant with respect to the natural action of $\Gamma$ on $S$, whose extension to the boundary has $F(\bdry S) = \Lambda(\Gamma)$.
\end{theorem}

Although this theorem certainly points in the direction of metric geometry, it does not strictly fall in this category. Indeed, the restriction of $\Gamma$ to fundamental groups of locally symmetric spaces and the use of $\dimup_H \bdry S$ in the rigidity inequality seem to place this result, in some sense, between Riemannian geometry and metric geometry.

In \cite{BK02}, M. Bonk and B. Kleiner extended the real-hyperbolic version of Bourdon's theorem to the metric setting. By real-hyperbolic, we mean the case that $(M^n,g_0)$ has constant sectional curvature $-1$, so that $S = \Hy^n$. Recall that $\bdry \Hy^n = \Sp^{n-1}$, which has Hausdorff dimension $n-1$.

\begin{theorem}[Bonk--Kleiner \cite{BK02}, \cite{BK04}] \label{BK}
Suppose that a group $\Gamma$ acts quasi-convex geometrically on a $\CAT(-1)$ metric space $X$. Let $n \geq 1$ be the topological dimension of $\Lambda(\Gamma)$. Then
$$\dimup_H \Lambda(\Gamma) \geq n,$$
and equality holds if and only if $\Gamma$ acts geometrically on an isometric copy of $\Hy^{n+1}$ in $X$.
\end{theorem}

The assertion $\dimup_H \Lambda(\Gamma) \geq n$ here is nothing special, as the Hausdorff dimension of any metric space is bounded from below by its topological dimension \cite[Chapter 7]{HW41}. Let us focus on the case of equality, then, and briefly describe the method of proof.

As in the rigidity theorems discussed above, the argument relies on a quasiconformal analysis of the limit set $\Lambda(\Gamma)$. The isometric action of $\Gamma$ on $X$ naturally passes to an action on $\Lambda(\Gamma)$ by uniformly quasi-M\"obius maps. As $\Gamma$ acts cocompactly on a quasi-convex subset of $X$, the induced action on $\Lambda(\Gamma)$ will be \textit{cocompact on triples}: any three distinct points in the limit set can be uniformly separated by applying an element of the group. This property should be viewed as a type of expanding dynamics on $\Lambda(\Gamma)$. It also allows us to conclude that $\Lambda(\Gamma)$ is \textit{Ahlfors regular} of dimension $n$: the $n$-dimensional Hausdorff measure of any metric ball $B(x,r)$ in the limit set is $\approx r^n$ (for $0\leq r \leq \diam \Lambda(\Gamma)$). 

The following theorem is the main result in \cite{BK02}.

\begin{theorem}[Bonk--Kleiner \cite{BK02}] \label{BKqm}
Let $Z$ be a compact, Ahlfors $n$-regular metric space with topological dimension $n \geq 1$. Suppose that $\Gamma \acts Z$ is a uniformly quasi-M\"obius group action that is cocompact on triples. Then $\Gamma \acts Z$ is quasisymmetrically conjugate to an action of $\Gamma$ on $\Sp^n$ by M\"obius transformations.
\end{theorem}

As M\"obius transformations can be extended naturally to isometries of $\Hy^{n+1}$, we obtain a geometric action of $\Gamma$ on $\Hy^{n+1}$. If $n \geq 2$, this puts us in the setting of Bourdon's theorem, which we apply to conclude that $\Gamma$ acts cocompactly on an isometric copy of $\Hy^{n+1}$ in $X$.

Actually, it turns out that appealing to Bourdon's theorem is not necessary. An alternative argument is given in \cite{BK04}, and it works just as well in the case that $n=1$.

\subsection{Rigidity on fractal spaces}

Following Bonk and Kleiner, this paper is primarily concerned with rigidity of expanding quasi-M\"obius group actions. Indeed, results in this setting often lead to rigidity theorems that are more geometric. Reconsidering, then, Theorem \ref{BKqm}, it is natural to wonder what one can say if the Hausdorff and topological dimensions differ.

A large collection of such examples are boundaries of Gromov hyperbolic groups equipped with a visual metric. In many important cases, the boundary is topologically a sphere, and always, it will be Ahlfors regular. Generally, though, the metric dimension is strictly larger than its topological dimension. In the case where the boundary is homeomorphic to $\Sp^2$, it is conjectured that there \textit{exists} an Ahlfors regular metric of dimension 2, but this is a difficult problem (see \cite[Section 5]{Bon06} for this formulation of Cannon's conjecture).

It is therefore of interest to obtain rigidity results for quasi-M\"obius group actions on fractal metric spaces---spaces in which the metric dimension differs from the topological dimension. This is the general objective in the present paper. In moving from such a broad goal to concrete theorems, we have kept an eye on applications to coarse hyperbolic geometry, which is a relevant setting for the study of Gromov hyperbolic groups. As a consequence, our main theorem will lead, via the work in \cite{BK02}, to an entropy rigidity result for geometric group actions on Gromov hyperbolic metric spaces with an asymptotic upper curvature bound. Naturally, this can be seen as a ``coarse" analog of the $\CAT(-1)$ rigidity theorem in \cite{BK02} and therefore also as an analog of Hamenst\"adt's theorem and of Bourdon's theorem (in the real-hyperbolic cases).

The precise statement of our main result is the following. We will discuss terminology and notation in subsequent sections, but let us make one important remark now. Rather than considering general quasi-M\"obius group actions, we restrict our attention to those that are \ti{strongly quasi-M\"obius}. In particular, each group element will act as a bi-Lipschitz homeomorphism. See Definition \ref{sqmdef} for a formal definition.

\begin{theorem} \label{mainthm}
Let $n \in \N$, $0<\e \leq 1$, and let $Z=(Z,d)$ be a compact metric space, homeomorphic to $\Sp^n$, and Ahlfors regular of dimension $n/\e$. Suppose that $\Gamma \acts Z$ is a strongly quasi-M\"obius action that is cocompact on triples. Assume, moreover, that $Z$ satisfies the following discrete length property:
\begin{equation} \label{length}
\begin{aligned}
\text{each } \delta \text{-path between two points }
x, y \text{ has length }\geq c \lp \tfrac{d(x,y)}{\delta} \rp^{1/\e}.
\end{aligned}
\end{equation}
Then there is a metric $d_{new}$ on $Z$ satisfying
$$C^{-1}d(x,y)^{1/\e} \leq d_{new}(x,y) \leq Cd(x,y)^{1/\e}$$
for some $C \geq 1$ and a bi-Lipschitz homeomorphism between $(Z,d_{new})$ and $\Sp^n$. Moreover, if $n \geq 2$, then this map can be taken to conjugate the action of $\Gamma$ on $Z$ to an action on $\Sp^n$ by M\"obius transformations.
\end{theorem}

\begin{remark}
The assumption that $Z$ is homeomorphic to $\Sp^n$ can be replaced by the assumption that $Z$ is an $n$-dimensional manifold. Indeed, in this case, the expanding behavior of the group action forces $Z$ to be a topological $n$-sphere. See, for example, the proof of Theorem 4.4 in \cite{KapB01}.
\end{remark}

\begin{remark}
Recall that if $\rho$ is a metric on $Z$, then $\rho^{\e}$ is also a metric whenever $0<\e \leq 1$. The metric spaces $(Z,\rho^{\e})$ are typically called ``snowflakes" of $(Z,\rho)$, in reference to the standard construction of the von Koch snowflake. In Theorem \ref{mainthm}, we go in the opposite direction, ``de-snowflaking" the original metric $d$ on $Z$ to a metric $d_{new}$ with better regularity.
\end{remark}

When the metric dimension and the topological dimension of $Z$ coincide (i.e., if $\e = 1$), the results in \cite{BK02} give a bi-Lipschitz homeomorphism between $Z$ and $\Sp^n$. Once these dimensions differ, relationships between the metric structures of $Z$ and $\Sp^n$ are more delicate. While Ahlfors regularity gives good control on volume, and the strongly quasi-M\"obius action provides robust self-similarity structure in $Z$, additional assumptions are needed to obtain rigidity statements. We impose the condition \eqref{length} because, in the case where $Z$ is the boundary of a hyperbolic metric space $X$, it arises naturally from upper curvature bounds on $X$.

In concise terms, the discrete length condition \eqref{length} is strong enough that it forces $(Z,d)$ to be a ``snowflake" of $\Sp^n$. Once we de-snowflake, we are able to pass almost directly through the theorem of Bonk and Kleiner. It is natural to ask, then, if there are weaker conditions one can place on $Z$ that still guarantee it is, say, quasisymmetrically equivalent to $\Sp^n$. This would be of significant interest, in particular for $n=2$.

As we suggested above, Theorem \ref{mainthm} leads to a rigidity theorem in a coarse geometric setting. The objects considered here are (Gromov) hyperbolic metric spaces with an appropriate asymptotic upper curvature bound. These curvature bounds, denoted by $\AC_u(\kappa)$, were introduced by M. Bonk and T. Foertsch \cite{BF06} as a coarse analog to the $\CAT(\kappa)$ conditions. We will discuss this further in Section \ref{rcg}, but for now we only mention that $\AC_u(-1)$ is an appropriate replacement for $\CAT(-1)$.

For hyperbolic metric spaces $X$, even with asymptotic upper curvature bounds, there is no canonical Hausdorff dimension of the boundary, as there was for $\CAT(-1)$-spaces. Indeed, the visual metrics on $\bdry X$ form a H\"older class, and there is not a natural choice of a bi-Lipschitz sub-class. Thus, to formulate an entropy-rigidity statement here, we must look back inside $X$ and use the coarse version of volume entropy---the same quantity that bridged the results of Hamenst\"adt and Bourdon. Namely, if $X$ is a hyperbolic metric space and $\Gamma$ acts on $X$, the exponential growth rate of the action is
$$e(\Gamma) = \limsup_{R \rightarrow \infty} \frac{\log(N(R))}{R},$$
where $N(R) = \# \{ \Gamma p \cap B_X(p,R) \}$ is the number of points in an orbit $\Gamma p$ of distance at most $R$ from $p$. Once again, the limit is independent of $p \in X$. We then have the corresponding coarse rigidity theorem.

\begin{theorem} \label{rigidity}
Let $X$ be a proper, geodesic, Gromov hyperbolic metric space, and let $\Gamma \acts X$ be a quasi-convex geometric group action. Suppose that $\Lambda(\Gamma)$ is homeomorphic to $\Sp^n$, with $n \geq 2$, and that there is an orbit $\Gamma p$ that is $\AC_u(-1)$. Then
$e(\Gamma) \geq n$
and equality holds if and only if there is a rough isometry $\Phi \colon \Hy^{n+1} \rightarrow \Gamma p$ that is roughly equivariant with respect to a geometric action of $\Gamma$ on $\Hy^{n+1}$.
\end{theorem}

\begin{remark}
Again, the assumption that $\Lambda(\Gamma)$ is a topological sphere can be weakened; it suffices to assume that $\Lambda(\Gamma)$ contains an open subset homeomorphic to $\R^n$. Indeed, this will imply that $\Lambda(\Gamma)$ is homeomorphic to $\Sp^n$ (cf.\ Theorem 4.4 in \cite{KapB01}). We leave as an open question, though, whether it suffices to assume only that the topological dimension of $\Lambda(\Gamma)$ is $n$.

Moreover, one should ask about the case $n=1$. We do not know if the conclusion in Theorem \ref{rigidity} holds in this case as well. See, however, Remark \ref{n=1r} at the end of Section \ref{rcg}, where we discuss what can be said in its place.
\end{remark}

This paper is organized as follows. In Section \ref{definitions}, we will introduce the necessary definitions, terminology, and background for the consideration and proof of Theorem \ref{mainthm}. Section \ref{sqm} will be devoted to a slightly technical study of strongly quasi-M\"obius group actions that will reveal some properties relevant for a ``de-snowflaking" result. In Section \ref{desnowflake}, we will state and prove this general de-snowflaking theorem, which forms the heart of the proof of Theorem \ref{mainthm}. In Section \ref{quant} we finish the proof of Theorem \ref{mainthm} by de-snowflaking and applying quantitative versions of theorems from \cite{BK02} and \cite{Tuk86}. We will also, of course, need to verify these quantitative versions. Finally, in Section \ref{rcg} we will prove Theorem \ref{rigidity} after discussing in more detail the terminology used in its statement.

\subsection*{Acknowledgements}
The author thanks Mario Bonk for many important insights and suggestions, and, generally, for introducing him to the wonderful relationship between hyperbolic and quasiconformal geometry. He is grateful to Ursula Hamenst\"adt and Peter Ha\"issinsky for helpful conversations during the IPAM workshop ``Interactions Between Analysis and Geometry." He also thanks Qian Yin and Humberto Silva Naves for useful discussions during the initial stages of this project. 

The author gratefully acknowledges partial support from NSF grant DMS-1162471.

\section{Definitions and Notation} \label{definitions}

Let $(Z,d)$ be a metric space. Occasionally, we will write $d_Z$ for the metric on $Z$ when this needs to be specified. If $x \in Z$ and $r>0$, then we use 
$$B(x,r) = \{ y \in Z : d(x,y) < r \}$$
to denote the open metric ball of radius $r$ about $x$.

If $x_1,x_2,x_3,x_4 \in Z$ are distinct points, we define their \ti{(metric) cross-ratio} as
$$[x_1,x_2,x_3,x_4] = \frac{d(x_1,x_3)d(x_2,x_4)}{d(x_1,x_4)d(x_2,x_3)}.$$
We are interested in maps between metric spaces that distort cross-ratios in a controlled manner. To make this precise, let $\eta \colon [0,\infty) \rightarrow [0,\infty)$ be a homeomorphism. Then a homeomorphism $f \colon X \rightarrow Y$ is called \ti{$\eta$-quasi-M\"obius} if 
$$[f(x_1),f(x_2),f(x_3),f(x_4)] \leq \eta([x_1,x_2,x_3,x_4])$$
for all distinct four-tuples $x_1,x_2,x_3,x_4 \in X$. Note that this definition makes sense for injective $f$ as well, but we will be concerned only with homeomorphisms in what follows.

A second class of maps that arise naturally in quasiconformal geometry are the quasisymmetric maps, which distort relative distances by a controlled amount. A homeomorphism $f \colon X \rightarrow Y$ is \ti{$\eta$-quasisymmetric} if 
$$\frac{d_Y(f(x_1),f(x_2))}{d_Y(f(x_1),f(x_3))} \leq \eta \lp \frac{d_X(x_1,x_2)}{d_X(x_1,x_3)} \rp$$
for all triples $x_1,x_2,x_3$ of distinct points in $X$.

The quasi-M\"obius and quasisymmetric conditions are closely related, though there are subtle differences. For example, every $\eta$-quasisymmetric map is $\tilde{\eta}$-quasi-M\"obius, where $\tilde{\eta}$ depends only on $\eta$. Conversely, if $X$ and $Y$ are bounded, then each individual $\eta$-quasi-M\"obius map will be $\tilde{\eta}$-quasisymmetric for \ti{some} $\tilde{\eta}$, but in general there is no quantitative relationship between $\eta$ and $\tilde{\eta}$. 

In this paper, we are mostly interested in studying metric spaces on which there is a group action by maps belonging to a particular function class. In such a context, the quasi-M\"obius and quasisymmetry conditions are very different. As quasi-M\"obius maps are the weaker of these two types, it makes sense to focus on these actions. This choice is further motivated by the following fact about hyperbolic groups (which occupy center stage in studying the geometry of hyperbolic metric spaces). If $G$ is a hyperbolic group and $\bdry G$ is its boundary (i.e., the Gromov boundary of the Cayley graph of $G$ with respect to a fixed finite generating set) equipped with a visual metric, then the isometric action of $G$ on its Cayley graph by translations extends to an action on $\bdry G$. Moreover, there is $\eta$ for which each $g \in G$ acts as an $\eta$-quasi-M\"obius map. Actually, something stronger is true: we can take $\eta$ to be linear (see Section \ref{rcg} for more details).

Quasi-M\"obius maps with a linear distortion function will play an important role in our analysis. Thus, we give them a name.

\begin{definition} \label{sqmdef}
A homeomorphism $f \colon X \rightarrow Y$ is called \ti{strongly quasi-M\"obius} if there is $C \geq 1$ for which
$$[f(x_1),f(x_2),f(x_3),f(x_4)] \leq C [x_1,x_2,x_3,x_4]$$
whenever $x_1,x_2,x_3,x_4 \in X$ are distinct.
\end{definition}

Each strongly quasi-M\"obius map between bounded metric spaces is actually bi-Lipschitz: there is a constant $C' \geq 1$ for which
$$\frac{1}{C'} d_X(x_1,x_2) \leq d_Y(f(x_1),f(x_2)) \leq C'd_X(x_1,x_2)$$
for all $x_1,x_2 \in X$ (see Remark \ref{sqmisbilip}). But again, the relationship between $C$ and $C'$ is not quantitative. We will study group actions by strongly quasi-M\"obius maps in much greater detail in subsequent sections.

Most of the group actions we encounter here will be of an expanding type, in the following sense.

\begin{definition} \label{coc}
An action of a group $\Gamma$ on a metric space $(Z,d)$ is said to be \ti{cocompact on triples} if there is $\delta > 0$ such that for every triple $x_1,x_2,x_3 \in Z$ of distinct points, there is a map $g \in \Gamma$ for which $d(gx_i,gx_j) \geq \delta$ if $i \neq j$.
\end{definition}

It should be no surprise that this assumption is again motivated by the geometry of hyperbolic groups: the action of a hyperbolic group on its boundary (equipped with a visual metric) is indeed cocompact on triples. More generally, the expanding behavior of a group action, combined with a (assumed) regularity of maps in the group, often translates into self-similarity properties of the metric space. See Lemma \ref{nearfar} for a particular manifestation of this principle.

A final metric property that will commonly undergird our spaces is a standard type of volume regularity.

\begin{definition}
A compact metric space $(Z,d)$ is \ti{Ahlfors $\alpha$-regular} (or Ahlfors regular of dimension $\alpha >0$) if there is a Borel measure $\mu$ on $Z$ and a constant $C \geq 1$ so that
\begin{equation} \label{Ahlfors}
\frac{1}{C}r^{\alpha} \leq \mu(B(x,r)) \leq Cr^{\alpha}
\end{equation}
for all $x \in Z$ and $0<r \leq  \diam Z$.
\end{definition}

Using standard covering arguments, it is not difficult to show that $Z$ is Ahlfors $\alpha$-regular if and only if \eqref{Ahlfors} holds with $\mu$ replaced by Hausdorff measure of dimension $\alpha$.

In subsequent sections, we will frequently encounter the $n$-dimensional sphere $\Sp^n$. Unless otherwise specified, we give it the chordal metric---the restriction of the Euclidean metric when $\Sp^n$ is viewed as the unit sphere in $\R^{n+1}$. However, every metric property of $\Sp^n$ that we consider will be preserved under a bi-Lipschitz change of coordinates. Thus, any metric that is bi-Lipschitz equivalent to the chordal metric would work just as well.

Finally, it will be convenient for us to suppress non-essential multiplicative constants in many inequalities. For quantities $A$ and $B$ that depend on some collection of input variables, we write $A \lesssim B$ to indicate that there is a constant $C$, independent of these variables, for which $A \leq C B$. When possible confusion could arise, we will indicate which data $C$ may depend on. For example, the bi-Lipschitz condition can be expressed simply as
$$d_X(x_1,x_2) \lesssim d_Y(f(x_1),f(x_2)) \lesssim d_X(x_1,x_2),$$
where the constants are uniform over all $x_1,x_2 \in X$.

\section{Strongly Quasi-M\"obius Group Actions} \label{sqm}

We now focus our attention on strongly quasi-M\"obius maps---those with a linear distortion function. Such maps tend to behave even more like traditional M\"obius functions than general quasi-M\"obius maps do. For example, each strongly quasi-M\"obius homeomorphism between bounded metric spaces is bi-Lipschitz (cf.\ Remark \ref{sqmisbilip}).

Strongly quasi-M\"obius maps are particularly important when they come in a group with uniform distortion constant. We will say that a group action $\Gamma \acts Z$ on a metric space $Z$ is \ti{strongly quasi-M\"obius} if there is a constant $C \geq 1$ for which every $g \in \Gamma$ is an $\eta$-quasi-M\"obius homeomorphism with $\eta(t) = Ct$.

The following lemma tells us that a strongly quasi-M\"obius group action that is cocompact on triples gives $Z$ locally self-similar structure: each ball can be blown up to a uniform scale by a homeomorphism that is essentially a scaling on that ball. See \cite[Section 2.3]{BS07} for a general discussion of local self-similarity in metric spaces. See also Lemma 5.1 in \cite{BK02} for a statement similar to ours, albeit in a slightly different context.

\begin{lemma} \label{nearfar}
Let $(Z,d)$ be a compact, connected metric space with at least two points, and let $\Gamma \acts Z$ be a strongly quasi-M\"obius group action which is cocompact on triples. For fixed $p \in Z$, $0 < r \leq \diam Z$, and $L \geq 2$, let $N=B(p,r)$ be a ``near" set and $F=Z\backslash B(p,Lr)$ be a ``far" set with respect to $p$. Then there is a map $g \in \Gamma$ satisfying the following:
\begin{enumerate}
\item[\textup{(i)}] $r \cdot d(x,y) \lesssim d(gx,gy) \lesssim (1/r) \cdot d(x,y)$ for all $x,y \in Z$,
\item[\textup{(ii)}] $(1/r) \cdot d(x,y) \lesssim d(gx,gy) \lesssim (1/r) \cdot d(x,y)$ for all $x,y \in N$,
\item[\textup{(iii)}] there exists $c > 0$ such that $B(gx,c) \subset gN$ for each $x \in B(p,r/2)$,
\item[\textup{(iv)}] $\diam gF \lesssim 1/L$.
\end{enumerate}
Here, the implicit constants and $c$ depend only on $\diam Z$, the constant $\delta$ in Definition \ref{coc}, and $C$ from the strongly quasi-M\"obius condition. In particular, they do not depend on $p$, $r$, or $L$.
\end{lemma}

Observe that property (ii) tells us that on $N$, the map $g$ is basically a scaling by $1/r$, in that $g$ blows up $B(p,r)$ to a uniform scale. Property (iii) guarantees that $gN$ will contain large balls around images of points that are well inside $N$. Or, to put it negatively, points outside of $N$ cannot get mapped nearby the images of points well within $N$. Property (iv) shows that if we take $L$ to be large, we can map the ``far" set $F$ to something negligible.

\begin{proof}
Given $p$, $r$, and $L$, we first choose three points that we wish to $\delta$-separate. Let $x_1=p$, and choose $x_2$ to be a point for which $d(x_1,x_2) = r/2$. Then choose $x_3$ so that $d(x_1,x_3) = r/4$. Such points $x_2$ and $x_3$ exist by the assumption that $Z$ is compact and connected. Take $g \in \Gamma$ so that $gx_1, gx_2, gx_3$ have pairwise distances at least $\delta$. Now that we have chosen $g$, we use $x'$ to refer to the image of points $x$ under $g$.

(i) Let $x,y \in Z$ with $x \neq y$. Then there are $i,j \in \{1,2,3\}$ for which $d(x,x_i) \geq r/8$ and $d(x,x_j) \geq r/8$. Of these, either $d(y,x_i) \geq r/8$ or $d(y,x_j) \geq r/8$; without loss of generality, say $d(y,x_i) \geq r/8$. We then have 
$$\frac{d(x',y')d(x_i',x_j')}{d(x',x_j')d(y',x_i')} \lesssim \frac{d(x,y)d(x_i,x_j)}{d(x,x_j)d(y,x_i)} \lesssim \frac{d(x,y) \cdot r}{r/8 \cdot r/8} \lesssim \frac{d(x,y)}{r} ,$$
and so 
$$d(x',y') \lesssim \frac{d(x',x_j')d(y',x_i')}{d(x_i',x_j')}\cdot \frac{d(x,y)}{r} \lesssim \frac{1}{\delta} \cdot \frac{d(x,y)}{r} \lesssim \frac{d(x,y)}{r} ,$$
which is the second inequality in (i). Recall that the implicit constant is allowed to depend on $\diam Z$ and on $\delta$.

For the first inequality, take $i,j \in \{1,2,3\}$ for which $d(x',x_i') \geq \delta/2$ and $d(x',x_j') \geq \delta/2$. Then either $d(y',x_i') \geq \delta/2$ or $d(y',x_j') \geq \delta/2$; without loss of generality, say $d(y',x_i') \geq \delta/2$. Then 
$$\frac{d(x',x_j')d(y',x_i')}{d(x',y')d(x_i',x_j')} \lesssim \frac{d(x,x_j)d(y,x_i)}{d(x,y)d(x_i,x_j)}  \lesssim \frac{1}{d(x,y) \cdot r} ,$$
and so 
$$d(x,y) \lesssim \frac{1}{r} \cdot \frac{d(x',y')d(x_i',x_j')}{d(x',x_j')d(y',x_i')} \lesssim \frac{d(x',y')}{r} ,$$
as desired.

\begin{remark} \label{sqmisbilip}
The same reasoning can be used to show that any strongly quasi-M\"obius map between bounded metric spaces is necessarily bi-Lipschitz. Indeed, notice that after choosing $g$, the arguments in (i) use only four facts: pairwise distances between $x_1$, $x_2$, and $x_3$ are $ \geq r/4$; pairwise distances between $x_1'$, $x_2'$, and $x_3'$ are $\geq \delta$; the domain and image of $g$ are both bounded; and $g$ is strongly quasi-M\"obius. In general, then, if $f$ is a strongly quasi-M\"obius homeomorphism between bounded metric spaces, choose distinct points $x_1$, $x_2$, and $x_3$ in the domain. These four facts will hold for some $r,\delta > 0$, so we can conclude that $f$ is bi-Lipschitz.
\end{remark}

(ii) Let $x,y \in N$. The second inequality here is directly from (i). For the first inequality, take $i,j$ for which $d(x',x_j'), d(y',x_i') \geq \delta/2$ as we did above. Then
$$\frac{d(x',x_j')d(y',x_i')}{d(x',y')d(x_i',x_j')} \lesssim \frac{d(x,x_j)d(y,x_i)}{d(x,y)d(x_i,x_j)}  \lesssim\frac{2r \cdot 2r}{d(x,y) \cdot r/4} \lesssim \frac{r}{d(x,y)}  ,$$
and so 
$$d(x,y) \lesssim r \cdot \frac{d(x',y')d(x_i',x_j')}{d(x',x_j')d(y',x_i')} \lesssim r \cdot d(x',y')  ,$$
as claimed.

(iii) Fix $x \in B(p,r/2)$, and note that if $y \in Z\backslash N$, then $d(x,y) \geq r/2$. Taking $i \in \{2,3\}$ for which $d(y',x_i') \geq \delta/2$, we have
$$\frac{d(x',p')d(y',x_i')}{d(x',y')d(p',x_i')} \lesssim \frac{d(x,p)d(y,x_i)}{d(x,y)d(p,x_i)}\lesssim \frac{d(x,p)}{r} \cdot \frac{d(y,x_i)}{d(x,y)}  .$$
As $d(x,y) \geq r/2$, we also have $d(y,x_i) \leq d(x,y) + d(x,x_i) \leq d(x,y) + 2r \leq 5d(x,y)$, and so
$$\frac{d(x',p')d(y',x_i')}{d(x',y')d(p',x_i')} \lesssim \frac{d(x,p)}{r}  .$$
Thus,
$$d(x',y') \gtrsim \frac{d(x',p')d(y',x_i')}{d(p',x_i')}\cdot \frac{r}{d(x,p)} \gtrsim \frac{d(x',p')}{d(x,p)} \cdot r \gtrsim 1$$
by the bounds we established in (ii). Let $c$ be the implicit constant in this last inequality. Then $B(x',c) \subset gN$ by the fact that $g$ is surjective.

(iv) We may, of course, assume that $B(p,Lr)$ is not all of $Z$. Then fix a point $x \in B(p,2Lr)\backslash B(p,Lr)$. We claim that $gF$ is contained in a small ball centered at $x'$. Indeed, let $y \in F$, and observe that
$$\begin{aligned}
&d(x,y) \leq d(x,x_1) + d(y,x_1) \leq 2Lr+d(y,x_1) \leq 3d(y,x_1), \\
&d(x_1,x_2) = r/2, \\
&d(x,x_1) \geq Lr, \\
&d(y,x_2) \geq d(y,x_1) - d(x_1,x_2) \geq d(y,x_1) - r .
\end{aligned}$$
Thus, we have 
$$\frac{d(x',y')d(x_1',x_2')}{d(x',x_1')d(y',x_2')} \lesssim \frac{d(x,y)d(x_1,x_2)}{d(x,x_1)d(y,x_2)} \lesssim \frac{d(y,x_1) \cdot r}{Lr \cdot (d(y,x_1)-r)}.$$
As $d(y,x_1) \geq Lr$ and the function $t \mapsto t/(t - r)$ is decreasing for $t > r$, we obtain
$$\frac{d(y,x_1)\cdot r}{Lr\cdot (d(y,x_1)-r)} \leq \frac{Lr}{L(Lr-r)} \lesssim \frac{1}{L}  ,$$ 
and so
$$d(x',y') \lesssim \frac{1}{L}\cdot \frac{d(x',x_1')d(y',x_2')}{d(x_1',x_2')} \lesssim \frac{1}{L}  .$$
Consequently, $gF$ is contained in a ball of radius $\lesssim L^{-1}$ centered at $x'$, as needed.
\end{proof}

The previous lemma gives us a good understanding of the type of expanding behavior found in a strongly quasi-M\"obius group action, when it acts cocompactly on triples. This use of a group element to ``blow up" a ball to a uniform scale is sometimes called a ``conformal elevator" (see, for example, \cite{HaPil10}). One can therefore view Lemma \ref{nearfar} as a particular type of conformal elevator that comes with a strongly quasi-M\"obius group action. This elevator will be essential in the proof of our de-snowflaking result, coming in the next section.

Actually, in the proof of that result, it is the conformal elevator itself (rather than the strongly quasi-M\"obius action generating it) that will be important. In order to work in greater generality, we make the following definition.

\begin{definition} \label{elevator}
A metric space $(Z,d)$ admits a \ti{conformal elevator} if there exists a constant $C$ and a function $\omega \colon (0,\infty) \rightarrow (0,\infty)$ with $\omega(t) \rightarrow 0$ as $t \rightarrow 0$ such that, for every choice of $p \in Z$, $0< r \leq \diam Z$, and $\lambda \geq 2$, there is a homeomorphism $g \colon Z \rightarrow Z$ with the following properties:
\begin{enumerate}
\item[\textup{(i)}] $d(gx,gy) \leq Cd(x,y)/r$ for all $x,y \in B(p,\lambda r),$
\item[\textup{(ii)}] $C^{-1} d(x,y)/r \leq d(gx,gy)$ for all $x,y \in B(p,r),$
\item[\textup{(iii)}] $B(gx,1/C) \subset g(B(p,r))$ for all $x \in B(p,r/C),$
\item[\textup{(iv)}] $\diam \lp Z \backslash g(B(p,\lambda r))\rp \leq \omega(1/\lambda)$.
\end{enumerate}
\end{definition}

The conclusions in Lemma \ref{nearfar} tell us that if $Z$ admits a strongly quasi-M\"obius group action that is cocompact on triples, then it admits a conformal elevator. We now turn our attention toward using the conformal elevator to de-snowflake a metric space.

\section{De-snowflaking} \label{desnowflake}

This section is devoted to establishing the following proposition, which provides quantitative conditions under which a metric space can be de-snowflaked by a particular amount.

\begin{prop} \label{mainprop}
Fix $n \in \N$ and $0<\e < 1$. Let $(Z,d)$ be a metric space with the following properties:
\begin{enumerate}
\item[\textup{(i)}] $Z$ is homeomorphic to $\Sp^n$,
\item[\textup{(ii)}] $Z$ admits a conformal elevator, in the sense of Definition \ref{elevator},
\item[\textup{(iii)}] every $\delta$-separated set in $Z$ has size at most $C\delta^{-n/\e}$,
\item[\textup{(iv)}] every discrete $\delta$-path from $x$ to $y$ in $Z$ has length at least \\$(1/C)\cdot(d(x,y)/\delta)^{1/\e}$.
\end{enumerate}
Then there is a metric $d_{new}$ on $Z$ satisfying
\begin{equation} \label{dnew}
d(x,y)^{1/\e} \lesssim d_{new}(x,y) \lesssim d(x,y)^{1/\e}
\end{equation}
where the implicit constant depends only on the data from assumptions \textup{(i)}--\textup{(iv)}.
\end{prop}

Recall that a ``$\delta$-separated set" is simply a set of points for which pairwise distances are at least $\delta$. Also, by a ``discrete $\delta$-path from $x$ to $y$" we mean a chain of points
$x=z_0, z_1, \ldots, z_l = y$
in $Z$ with $d(z_i,z_{i-1}) \leq \delta$. The length of such a chain is $l$; notice that this is one less than the number of points in the chain.

Some remarks on the assumptions in Proposition \ref{mainprop} are in order. The first condition gives $Z$ non-trivial topological structure and guarantees that the metric structure of $(Z,d)$ on large scales is similar to that of the standard sphere. The conformal elevator, as we have already said, gives us a way of moving from small scales to a uniformly large scale. The third assumption should be thought of as a volume condition; for example, it is easily implied by Ahlfors $n/\e$-regularity. In fact, our condition is similar to assuming that the Assouad dimension of $(Z,d)$ is $n/\e$. See \cite[Chapter 9]{BS07} for a discussion on various notions of metric dimension. Finally, condition (iv) is exactly the discrete length assumption that appears in Theorem \ref{mainthm}. It basically functions as a one-dimensional metric condition. Thus, the essential ingredients to our de-snowflaking result are topological regularity, metric self-similarity, upper bounds on volume, and lower bounds on one-dimensional metric structure.

Regarding the conclusion of the proposition, the ``data from assumptions (i)--(iv)" include the following: the parameters $\e$ and $n$; the diameter of $Z$; the constant $C$ from conditions (ii), (iii), and (iv); the function $\omega$ from the conformal elevator; and the modulus of continuity of a fixed homeomorphism between $Z$ and $\Sp^n$. We will also refer to these as the ``data associated to $Z$."

The proof of Proposition \ref{mainprop} will proceed as follows. We begin by fixing, for each length scale $e^{-\e k}$, a cover of $Z$ by metric balls. For $x,y \in Z$, the smallest number of these balls needed to join $x$ and $y$ provides a ``fuzzy" notion of distance at scale $e^{-\e k}$. After a proper normalization of this fuzzy distance, we let $k$ tend to infinity to obtain the metric $d_{new}$.

The lower bound in \eqref{dnew} will follow almost directly from the discrete length condition in (iv). The upper bound is more complicated, but in it we will see a nice interplay between the topological and metric structures of $Z$, which are linked together by the existence of the conformal elevator.

\subsection{The definition of $d_{new}$} \label{defdnew}

Fix notation as in the statement of the proposition. In particular, we let $C$ be a constant large enough so that conditions (iii) and (iv) hold, as well as the conditions from the definition of a conformal elevator. We also let $\omega$ be the function associated with the conformal elevator on $Z$. This notation will remain fixed throughout the proof.

By scaling the metric $d$, we may assume for simplicity that $\diam Z=1$. Indeed, the implicit constant in the desired conclusion is allowed to depend on the diameter, so we lose no generality.

For each $k \in \N$, fix a maximal $e^{-\e k}$-separated set in $Z$ and call it $P_k$. This also will remain fixed throughout the proof. It is not difficult to see that if $P$ is any $e^{-\e k}$-separated set contained in a ball $B(p,r)$, with $0<r\leq 1$, then
\begin{equation} \label{separatedset}
\# P \lesssim r^{n/\e} \cdot e^{nk}.
\end{equation}
Indeed, this follows from the volume bound in (iii) after applying the conformal elevator for $p$, $r$, and $\lambda=2$. In particular, if $r=e^{-\e m}$, then such a set has size $\lesssim e^{n(k-m)}$.

By maximality of $P_k$, we mean with respect to set inclusion. This is, of course, equivalent to $$Z= \bigcup_{x \in P_k} B(x,e^{-\e k})$$ for each $k$. It will be more convenient to work with balls of twice this radius, and so we refer to 
$$\{ B(x,2e^{-\e k}) : x \in P_k \}$$
as the set of \textit{$k$-balls}. For notational simplicity, we may abbreviate $B(x,2e^{-\e k})$ by $B_k(x)$ when $x \in P_k$. Often, the center-point $x$ of $B_k(x)$ is not important. As a result, we will usually denote $k$-balls simply by $B$ or $B_i$, e.g., when dealing with a chain of such balls. In these cases, $k$ will be understood from the context.

We first observe that for each $k$, the set of $k$-balls has controlled overlap, in that each $k$-ball intersects at most a uniformly bounded number of $k$-balls. Indeed, if $B_k(x)$ is a $k$-ball and $B_k(x_i)$, $1\leq i\leq m$, are those that intersect $B_k(x)$ non-trivially, then the collection $\{x_1,\ldots,x_m\}$ is an $e^{-\e k}$-separated set in the ball $B(x,4e^{-\e k})$. By (\ref{separatedset}) above, we get
\begin{equation} \label{overlap}
m \lesssim (4e^{-\e k})^{n/\e} \cdot e^{nk} \lesssim 1,
\end{equation}
where the implicit constant is allowed to depend on $n$ and $\e$.

A sequence of $k$-balls $B_1,B_2,\ldots,B_l$ with $B_i \cap B_{i+1} \neq \emptyset$ is called a \ti{$k$-ball chain}. We say that such a chain \ti{connects} two points $x$ and $y$ if $x \in B_1$ and $y \in B_l$. Observe that, as $B_i$ may not be a connected set itself, chains may not be connected topologically. This will pose no problem for our later analysis, though.

The length of a $k$-ball chain is simply the number of balls appearing in it, counted with multiplicity. For each $k$, let 
$$d_k(x,y) = \left( \text{length of shortest } k\text{-ball chain connecting } x \text{ and } y \right) \cdot e^{-k}.$$
The normalization by $e^{-k}$ is appropriate; indeed, each $k$-ball has diameter approximately $e^{-\e k}$ with respect to $d$, so its diameter with respect to the sought-after $d_{new}$ should be approximately $e^{-k}$. Note that $d_k$ is not actually a metric; for each $x$ we have $d_k(x,x) = e^{-k}$. But $d_k$ is symmetric and the triangle inequality clearly holds.

We now set $$d_{new}(x,y) = \limsup_{k \rightarrow \infty}d_k(x,y).$$ It is not immediate that this is a metric either. It is certainly symmetric, has $d_{new}(x,x)=0$ for all $x$, and satisfies the triangle inequality. The inequalities $0 < d_{new}(x,y) < \infty$ for $x \neq y$ will, however, be a consequence of proving that $d_{new}$ is bi-Lipschitz equivalent to $d^{1/\e}$. Thus, we wish to show that for each pair of distinct points $x,y \in Z$, $$d(x,y)^{1/\e} \lesssim d_k(x,y) \lesssim d(x,y)^{1/\e}$$ for $k$ large enough, where the implicit constants depend only on the data associated to $Z$.

As we mentioned before, the lower bound will be an easy consequence of assumption (iv) in the statement of the proposition---the lower bound on the length of discrete ``paths" between points. We quickly verify this.

Let $x,y \in Z$ be distinct and let $B_1,\ldots,B_l$ be a $k$-ball chain connecting $x$ and $y$. Then $x \in B_1$ and $y \in B_l$. Choose $x_i \in B_i \cap B_{i+1}$ for each $1\leq i \leq l-1$ and consider the discrete path
$$x=x_0,x_1,\ldots,x_{l-1},x_l = y$$
from $x$ to $y$. Observe that $d(x_i,x_{i+1}) \leq 4e^{-\e k}$, as $\diam B_i \leq 4e^{-\e k}$. Consequently,
$$l \gtrsim \lp \frac{d(x,y)}{4e^{-\e k}} \rp^{1/\e},$$
so that $l \gtrsim d(x,y)^{1/\e}\cdot e^k$ where the implicit constant depends only on $C$ and $\e$. This gives immediately that $d_k(x,y) \gtrsim d(x,y)^{1/\e}$, as desired.

We now turn to the upper bound, which is much more subtle. To obtain it, we will use the lower bound, along with a discrete length-volume inequality for cubes.

\subsection{The upper bound}
We begin by stating the crucial lemma, which will almost immediately give the upper bound when applied iteratively. This method of proof was motivated by a similar argument in \cite{Yin11}, where the author also sought to establish upper bounds on ``tile" chains connecting two points.

\begin{lemma} \label{iterlemma}
Let $x,y \in Z$ and $m \in \N$ with $d(x,y) \leq e^{-\e (m-1)}$. Then for each $k \geq m$, there is a $k$-ball chain connecting $B(x,e^{-\e m})$ and $B(y,e^{-\e m})$ of length at most $C'e^{k-m}$. Here, $C'$ depends only on the data associated to $Z$.
\end{lemma}

As should be clear, we say that a $k$-ball chain $B_1,\ldots,B_{\ell}$ connects two sets $A$ and $B$ if $B_1 \cap A$ and $B_{\ell} \cap B$ are non-empty. We will first see how this lemma implies the desired upper bound.

\begin{proof}[Proof of the Upper Bound] 
Suppose that Lemma \ref{iterlemma} holds. Fix distinct points $x$ and $y$ in $Z$, and choose $h \in \N$ so that $e^{-\e h} < d(x,y) \leq e^{-\e(h-1)}$. Recall that we have normalized $\diam Z = 1$. Also fix $k >h$; it is helpful to think of $k$ being very large relative to $h$. We will temporarily use $B_z^m$ to denote the ball $B(z,e^{-\e m})$ for $z \in X$ and $m \in \N$. This should not be confused with the shorthand notation we used earlier for $k$-balls.

By the lemma, there is a $k$-ball chain connecting $B_x^h$ and $B_y^h$ of length at most $C'e^{k-h}$. This gives us points $x=x_{1,0}$, $x_{1,1}$, $x_{1,2}$, $x_{1,3}=y$ where $x_{1,0}, x_{1,1} \in B_x^h$ and $x_{1,2}, x_{1,3} \in B_y^h$, and $x_{1,1}$ is connected to $x_{1,2}$ by a $k$-ball chain of length at most $C'e^{k-h}$.

We now iterate this process. Observe that $d(x_{1,0},x_{1,1}) \leq e^{-\e h}$ so that we can apply the lemma again to obtain a $k$-ball chain of length at most $C'e^{k-(h+1)}$ connecting $B_{x_{1,0}}^{h+1}$ and $B_{x_{1,1}}^{h+1}$. This gives us points $x=x_{2,0}$, $x_{2,1}$, $x_{2,2}$, $x_{2,3}=x_{1,1}$ where $x_{2,0}, x_{2,1} \in B_{x_{1,0}}^{h+1}$ and $x_{2,2}, x_{2,3} \in B_{x_{1,1}}^{h+1}$, and $x_{2,1}$ is connected to $x_{2,2}$ by a $k$-ball chain of length at most $C'e^{k-(h+1)}$. Of course, we do a similar process to the pair of points $x_{1,2}$ and $x_{1,3}$.

The $m$th step in this process (for $1\leq m \leq k-h$) proceeds as follows. From the $(m-1)$-th step, we have $2^m$ points $$x=x_{m-1,0},x_{m-1,1},\ldots,x_{m-1,2^m-1}=y$$
satisfying $d(x_{m-1,i},x_{m-1,i+1}) \leq e^{-\e (h+m-2)}$ for each even integer $i \in \{0,\ldots,2^m-2\}$. Moreover, there is a previously-constructed $k$-ball chain connecting $x_{m-1,i+1}$ to $x_{m-1,i+2}$.

It will be convenient to rename these $2^m$ points so that they appear in the $m$th step. To do this, we let 
$$x_{m,j}=
\begin{cases}
x_{m-1,j/2} , & \text{ if } j \equiv 0 \text{ mod } 4 \\
x_{m-1,(j-1)/2} , & \text{ if } j \equiv 3 \text{ mod } 4
\end{cases}$$
for $0\leq j \leq 2^{m+1}-1$. We do not yet define the points corresponding to $j \equiv 1,2$ mod $4$, because we still have to find them. 

To this end, observe that $d(x_{m,j},x_{m,j+3}) \leq e^{-\e (h+m-2)}$ for each $j \equiv 0$ mod $4$. Thus, applying Lemma \ref{iterlemma} to these points, we find a $k$-ball chain of length at most $C'e^{k-(h+m-1)}$ connecting the balls $B_{x_{m,j}}^{h+m-1}$ and $B_{x_{m,j+3}}^{h+m-1}$. We can therefore choose points $x_{m,j+1} \in B_{x_{m,j}}^{h+m-1}$ and $x_{m,j+2} \in B_{x_{m,j+3}}^{h+m-1}$ that are connected by this $k$-ball chain.

We have now obtained $2^{m+1}$ points $$x=x_{m,0},x_{m,1},\ldots,x_{m,2^{m+1}-1}=y,$$
such that $$d(x_{m,i},x_{m,i+1}) \leq e^{-\e (h+m-1)}$$
for each even integer $i \in \{0,\ldots,2^{m+1}-2\}$. Note that we have constructed $2^{m-1}$ different $k$-ball chains at this step, each of length at most $C' e^{k-(h+m-1)}$. Moreover, for each even $i \in \{0,\ldots,2^{m+1}-2\}$, there is a $k$-ball chain connecting $x_{m,i+1}$ to $x_{m,i+2}$, constructed either at this $m$th step or at a previous one. 

Consider what happens at the end of the $(k-h+1)$-th step. We obtain $2^{k-h+2}$ points $$x=x_{k-h+1,0},x_{k-h+1,1},\ldots,x_{k-h+1,2^{k-h+2}-1}=y$$ satisfying $d(x_{k-h+1,i}, x_{k-h+1,i+1}) \leq e^{-\e k}$ for each even $i$. Consequently, there is a $k$-ball containing both $x_{k-h+1,i}$ and $x_{k-h+1,i+1}$. Indeed, if $z \in P_k$ with $d(z,x_{k-h+1,i}) < e^{-\e k}$, then 
$$x_{k-h+1,i}, x_{k-h+1,i+1} \in B(z,2e^{-\e k}) = B_k(z).$$
Moreover, there is a $k$-ball chain connecting $x_{k-h+1,i+1}$ to $x_{k-h+1,i+2}$ that was constructed at some step of the whole process.

Concatenating these $k$-ball chains, using the single $k$-balls containing $x_{k-h+1,i}$ and $x_{k-h+1,i+1}$ to join them together, we end up with a $k$-ball chain from $x=x_{k-h+1,0}$ to $y=x_{k-h+1,2^{k-h+2}-1}$ of length at most 
$$2^{k-h+1} + \sum_{m=1}^{k-h+1} 2^{m-1}\cdot C'e^{k-(h+m-1)} \lesssim e^{k-h}.$$
By the way we chose $h$, we obtain
$$d_k(x,y) \lesssim e^{k-h}\cdot e^{-k} \lesssim \left( e^{-\e h} \right)^{1/\e} \lesssim d(x,y)^{1/\e},$$
where the implicit constants depend only on $C'$.
\end{proof}

It therefore remains to prove Lemma \ref{iterlemma}. Broadly, our goal is to use the conformal elevator on $Z$ to blow up the balls $B(x,e^{-\e m})$ and $B(y,e^{-\e m})$ to a uniform scale, so that we can essentially reduce to the case that $m \approx 1$. Establishing the analogous bound on this uniform scale will require some topological arguments, combined with the discrete volume and length bounds on $Z$. We will first develop the topological tools necessary to carry this out.

\subsection{A discrete length-volume inequality}

An important topic in metric geometry is the relationship between the volume of a space and the lengths of curves that, in some way, generate it. See \cite[Chapter 4]{Grom99} for a survey of methods and results in this spirit. Among these is the following theorem, originally proved by W. Derrick \cite{Der69}. We state it in the form cited in \cite{Grom99} in order to motivate more clearly what will follow.

\begin{theorem}[Derrick {\cite[Theorem 3.4]{Der69}}]
Let $g$ be a Riemannian metric on the cube $[0,1]^n$, and let $F_k,G_k$, $1 \leq k \leq n$, denote the pairs of opposite codimension-1 faces of $[0,1]^n$. Let $d_k$ be the distance between $F_k$ and $G_k$ with respect to the metric $g$. Then $$\Vol(g) \geq d_1d_2\cdots d_n.$$
\end{theorem}

We will need a discrete/topological version of this theorem. Incidentally, the proof of the discrete version mimics the proof of the Riemannian version.

To set this up, let $U_1,\ldots,U_N$ be an open cover of the cube $[0,1]^n$. Again let $F_k, G_k$ denote the pairs of opposite codimension-1 faces: $\pi_k(F_k) = \{0\}$ and $\pi_k(G_k) = \{1\}$, where $\pi_k \colon \R^n \rightarrow \R$ is the projection onto the $k$th coordinate axis. We say that $U_{i_1},\ldots,U_{i_l}$ is a chain if $U_{i_j} \cap U_{i_{j+1}} \neq \emptyset$ for each $j$. Moreover, such a chain is said to connect two sets $A$ and $B$ if $U_{i_1} \cap A \neq \emptyset$ and $U_{i_l} \cap B \neq \emptyset$.

\begin{prop} \label{cubes}
Let $U_1,\ldots,U_N$ be as above, and let $d_k$ denote the smallest number of sets $U_i$ in a chain that connects $F_k$ and $G_k$. Then $$N \geq d_1d_2\cdots d_n.$$
\end{prop}

\begin{proof}
Without loss of generality, we may assume that no $U_i$ is redundant, i.e., that for each $i$, there is a point $x_i \in U_i$ that belongs to no other $U_j$. Otherwise, we could delete $U_i$ from the cover, thereby decreasing the total number of sets without reducing the numbers $d_k$.

We first define a map $f_0 \colon \{x_1,\ldots,x_N\} \rightarrow \Z^n$ where the $k$th component is
$$\pi_k(f_0(x_i)) = 
\begin{array}{l}
\text{ minimal number of sets } U_j \text{ in a chain} \\
\text{ that connects } F_k \text{ and } \{x_i\}.
\end{array}$$
The existence of such a chain follows from the fact that there is a path in $[0,1]^n$ from $F_k$ to $x_i$ and the collection $\{U_j\}$ is an open cover of this path. 

Now we extend $f_0$ to a map on $[0,1]^n$ by using a partition of unity subordinate to $\{U_i\}$. More precisely, let $\{ \varphi_i \}$ be a partition of unity such that $\supp(\varphi_i) \subset U_i$, and let 
$$f(x) = \sum_{i=1}^{N} \varphi_i(x)f_0(x_i) = \sum_{i=1}^{N} \varphi_i(x)y_i,$$
for $x \in [0,1]^n$, where we use $y_i$ to denote $f_0(x_i)$. Observe that $f$ does indeed extend $f_0$ because $\varphi_j(x_i) =0$ for $j \neq i$ and $\varphi_i(x_i) = 1$. It is also, of course, continuous.

We claim that $\pi_k( f(F_k) ) = \{1\}$ and $\pi_k( f(G_k) ) \subset [d_k,\infty)$ for each $k$. For $x \in F_k$, let $U_{i_1},\ldots,U_{i_m}$ be the sets containing $x$ so that
$$f(x) = \sum_{j=1}^{m} \varphi_{i_j}(x)y_{i_j}.$$
As $x \in U_{i_j} \cap F_k$, we see that the single set $U_{i_j}$ connects $F_k$ and $\{x_{i_j}\}$. Thus, the $k$th coordinate of $y_{i_j}$ is 1. Consequently,
$$\pi_k( f(x) ) = \sum_{j=1}^m \varphi_{i_j}(x) = 1.$$
Similarly, if $x \in G_k$, again let $U_{i_1},\ldots,U_{i_m}$ be the sets containing $x$. Then $U_{i_j} \cap G_k$ is non-empty, so any chain connecting $F_k$ to $\{x_{i_j}\}$ (which necessarily must end with the set $U_{i_j}$) actually connects $F_k$ and $G_k$. By the definition of $d_k$, this chain has size at least $d_k$. Thus, the $k$th coordinate of $y_{i_j}$ is at least $d_k$. As a result,
$$\pi_k( f(x) )= \sum_{j=1}^m \varphi_{i_j}(x) \pi_k(y_{i_j}) \geq d_k \sum_{j=1}^m \varphi_{i_j}(x) = d_k.$$

We now claim that the image of $f$ must contain the $n$-dimensional rectangle $$S = \prod_{k=1}^n [1,d_k].$$ If not, there exists a point $y \in S \backslash f([0,1]^n)$. As $f$ is continuous, $f([0,1]^n)$ is closed, so we may assume that $y$ is in the interior of $S$. Let $g$ be a homeomorphism from $[0,1]^n$ to $S$ that sends corresponding faces to corresponding faces (an affine map will do). Then by the previous claim,
$$f_t = (1-t) f|_{\partial [0,1]^n} + t g|_{\partial [0,1]^n}$$ gives a homotopy with values in $\R^n \backslash \inter(S)$. In particular, $f|_{\partial [0,1]^n}$ is homotopic to $g|_{\partial [0,1]^n}$ in $\R^n \backslash \{y\}$. 

Fix a simplicial decomposition of $[0,1]^n$; this gives a corresponding decomposition of $\partial [0,1]^n$. The latter decomposition allows us to express $f|_{\partial [0,1]^n}$ and $g|_{\partial [0,1]^n}$ as singular $(n-1)$-chains with integer coefficients. Abusing notation, we continue to denote the chains by $f|_{\partial [0,1]^n}$ and $g|_{\partial [0,1]^n}$. The homotopy given above implies that the corresponding classes $[f|_{\partial [0,1]^n}]$ and $[g|_{\partial [0,1]^n}]$ are equal in the singular homology group $H_{n-1}(\R^n \backslash \{y\})$. Notice, though, that $f|_{\partial [0,1]^n}$ extends to the map $f \colon [0,1]^n \rightarrow \R^n \backslash \{y\}$, which we can view as a singular $n$-chain via the decomposition of $[0,1]^n$. Thus, the chain $f|_{\partial [0,1]^n}$ is the image of the chain $f$ under the boundary map, so $[f|_{\partial [0,1]^n}]$ is zero in $H_{n-1}(\R^n \backslash \{y\})$. In particular, $[g|_{\partial [0,1]^n}]$ is also zero. This, however, contradicts the fact that $[g|_{\partial [0,1]^n}]$ generates $H_{n-1}(\R^n \backslash \{y\})$, which is isomorphic to $\Z$. Hence, it must be that $S \subset f([0,1]^n)$.

Finally, we claim that if $f(x) \in \Z^n$, then $f(x) = y_i$ for some $i$. Let $U_{i_1},\ldots,U_{i_m}$ be the sets for which $\varphi_{i_j}(x) > 0$, so that $x \in U_{i_1} \cap \cdots \cap U_{i_m}$. Then 
$$f(x) = \sum_{j=1}^{m} \varphi_{i_j}(x)y_{i_j},$$
and for each $j,l \in \{1,\ldots,m\}$, $$|| y_{i_j} - y_{i_l} ||_{\infty} \leq 1,$$
where $|| y ||_{\infty} = \max \{| \pi_k(y)| : 1\leq k \leq n\}$ is the $\ell^{\infty}$-norm.
This inequality follows immediately from the fact that $U_{i_j}$ and $U_{i_{\ell}}$ have non-trivial intersection. Consequently, for each $k$, there is an integer $a_k$ such that 
$$\pi_k(y_{i_j}) \in \{a_k, a_k+1\}$$ 
for all $j$; namely, $a_k = \min \{ \pi_k(y_{i_j}) : 1\leq j \leq m\}$. Now fix $k$, and let 
$$I = \{ j : \pi_k(y_{i_j}) = a_k\} \text{ and } J = \{ j : \pi_k(y_{i_j}) = a_k+1\}.$$ 
By the definition of $a_k$, we have $I \neq \emptyset$. Then
$$\begin{aligned}
\pi_k(f(x)) &= \sum_{j=1}^m \varphi_{i_j}(x) \pi_k(y_{i_j}) = \sum_{j \in I} \varphi_{i_j}(x)a_k + \sum_{j \in J} \varphi_{i_j}(x)(a_k+1) \\
&= a_k + \sum_{j \in J} \varphi_{i_j}(x).
\end{aligned}$$
By assumption, $\pi_k(f(x))$ is an integer, so $\sum_{j \in J} \varphi_{i_j}(x)$ is also an integer, necessarily equal to 0 or 1. This can happen only if $J = \emptyset$ or $J = \{1,\ldots, m\}$, but the latter implies that $I=\emptyset$, contrary to assumption. Thus, $J=\emptyset$, so each $y_{i_j}$ has $k$th coordinate $a_k$. In particular, $f(x) = (a_1,\ldots,a_n)$ as well, giving $f(x) = y_{i_1}$.

From the previous claims we obtain the desired conclusion immediately. Each integer lattice point in the cube $S = [1,d_1]\times \cdots \times [1,d_n]$ has some $x_i$ in its pre-image under $f$. There are $d_1d_2\cdots d_n$ integer lattice points in $S$, so $N \geq d_1d_2\cdots d_n$.
\end{proof}

It is easy to see that Proposition \ref{cubes} still holds if all of the sets $U_i$ are assumed to be closed. Indeed, we can enlarge $U_i$ by a small amount to obtain open sets $U_i'$ without changing the incidence structure. Then apply the proposition to these open sets.

More importantly for our later use, we point out that the proposition above remains true if we replace $[0,1]^n$ by a topological cube. Indeed, the assumptions and conclusions are entirely topological.

Before moving on to the proof of Lemma \ref{iterlemma}, we must establish a basic fact about finding topological cubes in the sphere $\Sp^n$.

\begin{lemma} \label{thereisacube}
Let $B_0$ and $B_1$ be metric balls of radius $\delta >0$ in $\Sp^n$ for which $\dist(B_0,B_1) \geq \delta$. Suppose that $E \subset \Sp^n$ has $\diam(E) < \delta$ and $\dist(B_i,E) \geq \delta$ for $i=0,1$. Then there is a set $S \subset \Sp^n$ with the following properties:
\begin{enumerate}
\item[(i)] $S$ is homeomorphic to $[0,1]^n$;
\item[(ii)] the faces $\{0\} \times [0,1]^{n-1}$ and $\{1\} \times [0,1]^{n-1}$ correspond, under this homeomorphism, to sets $C_0 \subset B_0$ and $C_1 \subset B_1$, respectively;
\item[(iii)] if $x$ and $y$ in $S$ correspond to points that lie in opposite codimension-1 faces of $[0,1]^n$, then $d_{\Sp^n}(x,y) \geq c \delta^3$;
\item[(iv)] $S$ is disjoint from $E$.
\end{enumerate}
Here, $c$ is an absolute constant.
\end{lemma}

We should remark that the bound in (iii) is certainly far from optimal. It is, however, sufficient for our purposes and makes the following proof much simpler.

\begin{proof}
By rotation, we may assume that $E$ contains the north pole $N = (0,\ldots,0,1)$. Let $D$ be the metric ball of radius $\delta$ centered at $N$ so that $E \subset D$ and $B_i \cap D = \emptyset$ for $i=0,1$.

Let $p \colon \Sp^n \to \R^n$ be the stereographic projection 
$$(x_1,\cdots,x_{n+1}) \mapsto \lp \frac{x_1}{1-x_{n+1}},\ldots,\frac{x_n}{1-x_{n+1}} \rp.$$ 
Then $p(\partial D)$ is an $(n-1)$-dimensional sphere of radius $R \in [1/\delta,2/\delta]$. Moreover, we can say that $p|_{\Sp^n\backslash D}$ is a bi-Lipschitz map onto the Euclidean ball $B_{\R^n}(0,R)$ with
\begin{equation} \label{stereo}
d_{\Sp^n}(x,y) \lesssim d_{\R^n}(p(x),p(y)) \lesssim \frac{1}{\delta^2} d_{\Sp^n}(x,y),
\end{equation}
where the implicit constants are absolute. This follows from trigonometric arguments in the plane containing $N$, $p(x)$, and $p(y)$. As a result, $p(B_0)$ and $p(B_1)$ are Euclidean balls in $B_{\R^n}(0,R)$ with
$$\dist(p(B_0),p(B_1)) \gtrsim \delta \hspace{0.3cm} \text{and} \hspace{0.3cm} \diam(p(B_i)) \gtrsim \delta$$
for $i=0,1$. It is then easy to find an $n$-dimensional topological cube $\hat{S} \subset B_{\R^n}(0,R)$ with a pair of opposite codimension-1 faces in $p(B_0)$ and $p(B_1)$, respectively, and for which any pair of opposite faces are at distance $\gtrsim \delta$.

Now let $S=p^{-1}(\hat{S})$. Properties (i), (ii), and (iv) immediately follow from our choice of $\hat{S}$, and property (iii) is a consequence of the bounds in \eqref{stereo}.
\end{proof}

\subsection{Proof of Lemma \ref{iterlemma}}

We will prove a slight variant of the lemma, which easily implies the form stated above. Namely, we show that if $x$ and $y$ are distinct with $e^{-\e m} < d(x,y) \leq e^{-\e(m-1)}$, then for each $k \geq m$ there is a $k$-ball chain connecting the balls $B(x,e^{-\e m})$ and $B(y,e^{-\e m})$ of length at most $C'e^{k-m}$. 

It is straightforward to obtain Lemma \ref{iterlemma} from this. Indeed, let $x,y \in Z$ and $m \in \N$ with $d(x,y) \leq e^{-\e(m-1)}$. Fix $k \geq m$, and let $m' \geq m$ be the integer for which $e^{-\e m'} < d(x,y) \leq e^{-\e(m'-1)}$. If $k \geq m'$, then the desired conclusion in Lemma \ref{iterlemma} follows immediately from the conclusion of the variant. If $k < m'$, then $d(x,y) \leq e^{-\e k}$ so that $x$ and $y$ are contained in a common $k$-ball. Thus, it suffices to prove the variant.

To this end, let $x,y \in Z$ with $e^{-\e m} < d(x,y) \leq e^{-\e(m-1)}$, and fix $k \geq m$. For ease of notation, let
$$B_x = B(x,e^{-\e m}) \hspace{0.5cm} \text{and} \hspace{0.5cm} B_y = B(y,e^{-\e m}),$$
which again should not be confused with the earlier notation for $k$-balls. To find a short $k$-ball chain connecting $B_x$ and $B_y$, we will proceed in the following way. We first restrict our attention to larger balls
$$B(p,r) \subset B(p,\lambda r),$$
containing both $B_x$ and $B_y$ but still of radius roughly $e^{-\e m}$. By estimates we have discussed earlier, such a ball intersects $\lesssim e^{n(k-m)}$ $k$-balls. Applying the conformal elevator at this location and scale enlarges $B_x$ and $B_y$ to a uniform size so that we can find a ``wide" topological cube inside the image of $B(p,\lambda r)$. This cube will have a pair of opposite codimension-1 faces in the images of $B_x$ and $B_y$, respectively, and each pair of opposite faces will be uniformly far apart. Pulling this cube back down to scale $\approx e^{-\e m}$, it will be covered by those $k$-balls that intersect $B(p,\lambda r)$. This puts us in the setting of Proposition \ref{cubes}, where the lower discrete length bound will imply that $d_i \gtrsim e^{k-m}$ for each $i$. As the size of this cover is $\lesssim e^{n(k-m)}$, Proposition \ref{cubes} guarantees that $d_i \lesssim e^{k-m}$ as well. In particular, there is a chain connecting $B_x$ and $B_y$ of length roughly $e^{k-m}$.

We must, of course, make these arguments rigorous; to do so, it will be convenient to set $p=x$ and $r=2Ce^{-\e(m-1)}$, where $C$ is the large constant we chose at the beginning of Section \ref{defdnew}. Observe then that
$$B_x \cup B_y \subset B(p,r),$$
and moreover, that $x,y \in B(p,r/C)$. Let us also choose $\lambda$ (for later use in applying the conformal elevator) in the following way. Fix a homeomorphism $F \colon Z \rightarrow \Sp^n$ and let $0<\delta<1$ be small enough that 
\begin{equation} \label{deltachoice}
d(z,w) \geq \frac{1}{2C^2e^{\e}} \hspace{0.5cm} \text{implies} \hspace{0.5cm} d_{\Sp^n}(F(z),F(w)) \geq 3\delta
\end{equation}
for $z,w \in Z$. Then take $\lambda \geq 2$ large enough so that
\begin{equation} \label{lambdachoice}
d(z,w) < \omega \lp \tfrac{1}{\lambda} \rp \hspace{0.5cm} \text{implies} \hspace{0.5cm} d_{\Sp^n}(F(z),F(w)) < c \delta^3,
\end{equation}
where $0<c<1$ is the constant from Lemma \ref{thereisacube}. Note that $\lambda$ will depend on the modulus of continuity of $F$ and $F^{-1}$.

The conformal elevator on $Z$ gives a map $g$ for this choice of $p$, $r$, and $\lambda$. If it happens that $r > 1$ (i.e., if $m$ is small), then we simply choose $g$ to be the identity map. All of the following estimates work equally well in this case.

Let $x' = g(x)$, $y'=g(y)$, and $K = g(Z\backslash B(p,\lambda r))$. Property (ii) of the conformal elevator guarantees that
$$d(x',y') \geq \frac{d(x,y)}{Cr} \geq \frac{e^{-\e m}}{2C^2e^{-\e(m-1)}} = \frac{1}{2C^2e^{\e}},$$
so $x'$ and $y'$ are far apart. Property (iii) tells us that
\begin{equation} \label{bb}
B(x',1/C) \cup B(y',1/C) \subset g(B(p,r)),
\end{equation}
and as $g$ is a homeomorphism, this implies that
$$\dist(\{x',y'\},K) \geq \tfrac{1}{C}.$$
Moreover, we claim that
\begin{equation} \label{contained}
B\lp x', \tfrac{1}{2C^2e^{\e}}\rp \subset g(B_x) \hspace{0.5cm} \text{and} \hspace{0.5cm} B \lp y', \tfrac{1}{2C^2e^{\e}} \rp \subset g(B_y).
\end{equation}
Indeed, if $w \in B(x',1/(2C^2e^{\e}))$, then $z=g^{-1}(w)$ must be in $B(p,r)$ by (\ref{bb}). Consequently, property (ii) again gives
$$d(x,z) \leq Cr\cdot d(gx,gz) = Cr\cdot d(x',w) < \frac{2C^2e^{-\e(m-1)}}{2C^2e^{\e}} = e^{-\e m},$$
so that $z \in B_x$. Hence, $w=g(z)$ is in $g(B_x)$. The same reasoning works also for $B_y$. Finally, property (iv) of the conformal elevator guarantees that
$$\diam K \leq \omega \lp \tfrac{1}{\lambda} \rp,$$
which we view as being very small.

Let us now use the homeomorphism $F \colon Z \rightarrow \Sp^n$ to ``regularize" this large-scale configuration. By our choice of $\delta$ from (\ref{deltachoice}), we have
$$d_{\Sp^n} \lp F(x'),F(y')\rp \geq 3\delta \hspace{0.5cm} \text{and} \hspace{0.5cm} \dist \lp \{F(x'),F(y')\},F(K) \rp \geq 3\delta$$
so that $B_0= B_{\Sp^n}(F(x'),\delta)$ and $B_1=B_{\Sp^n}(F(y'),\delta)$ are metric balls in $\Sp^n$ with 
$$\dist(B_0,B_1) \geq \delta$$ 
and of distance at least $\delta$ from $F(K)$. Moreover,
observe that 
\begin{equation} \label{B0}
B_0 \subset F \lp B(x', 1/(2C^2e^{\e})) \rp \subset F \lp g(B_x) \rp
\end{equation}
and
\begin{equation} \label{B1}
B_1 \subset F \lp B(y',1/(2C^2e^{\e})) \rp \subset F \lp g(B_y) \rp,
\end{equation}
both of which follow from (\ref{deltachoice}) and (\ref{contained}). Also note that by our choice of $\lambda$, 
$$\diam F(K) \leq c \delta^3 < \delta.$$

The metric balls $B_0$ and $B_1$ and the set $F(K)$ therefore satisfy the hypotheses in Lemma \ref{thereisacube}. Let $\hat{S} \subset \Sp^n \backslash F(K)$ be the $n$-dimensional topological cube given in the conclusion of this lemma. Then $\hat{S}$ has a pair of opposite codimension-1 faces $\hat{C}_0$ and $\hat{C}_1$ in $B_0$ and $B_1$, respectively; moreover, any two opposite faces have spherical distance $\geq c \delta^3$ from each other.

Now send this set $\hat{S}$ back to $Z$ via the homeomorphism $(F \circ g)^{-1}$; that is, let
$$S = g^{-1}\circ F^{-1}(\hat{S})$$
so that $S$ is a topological cube in the ball $B(p,\lambda r)$. Observe that it has a pair of opposite codimension-1 faces
$$C_0 = g^{-1}\circ F^{-1}(\hat{C}_0) \hspace{0.5cm} \text{and} \hspace{0.5cm} C_1 = g^{-1}\circ F^{-1}(\hat{C}_1) $$
that lie within $B_x$ and $B_y$, respectively. This follows from the inclusions in \eqref{B0} and \eqref{B1}.

Consider the set of $k$-balls that meet $B(p,\lambda r)$. Intersect each $k$-ball with $S$, and call the resulting collection $\mathcal{U}$. The estimate in (\ref{separatedset}) implies that 
$$\# \mathcal{U} \lesssim e^{n(k-m)}.$$
Hence, $\mathcal{U}$ is an open cover of the topological cube $S$ by $\lesssim e^{n(k-m)}$ sets. In view of Proposition \ref{cubes}, we wish to show that each chain from $\mathcal{U}$ that joins opposite codimension-1 faces of $S$ must have $\gtrsim e^{k-m}$ sets.

To this end, let $U_1,\ldots,U_l$ be such a chain, so that $U_i \cap U_{i+1} \neq \emptyset$ for each $i$, and there are $a \in U_1$ and $b \in U_l$ in opposite faces of $S$. As $F\circ g(a)$ and $F\circ g(b)$ lie in opposite faces of $\hat{S}$, we know that
$$d_{\Sp^n}(F\circ g(a),F\circ g(b)) \geq c \delta^3.$$
By our choice of $\lambda$ in (\ref{lambdachoice}), this implies that
$$d(ga,gb) \geq \omega \lp \tfrac{1}{\lambda} \rp.$$
Property (i) of the conformal elevator then guarantees that
$$d(a,b) \geq \frac{r \cdot d(ga,gb)}{C} = 2e^{-\e(m-1)}d(ga,gb) \gtrsim e^{-\e m},$$
because $\omega(1/\lambda)$ is a uniform constant. The points
$$a=x_0,x_1,\ldots,x_{l-1},x_l = b$$
where $x_i \in U_i \cap U_{i+1}$ for each $1 \leq i \leq l-1$, form a discrete $4e^{-\e k}$-path from $a$ to $b$. Consequently,
$$l \gtrsim \lp \frac{d(a,b)}{4e^{-\e k}} \rp^{1/\e} \gtrsim e^{k-m},$$
as desired.

Using the notation from Proposition \ref{cubes}, let $d_1$ denote the smallest number of sets in $\mathcal{U}$ that form a chain connecting $C_0$ to $C_1$. Similarly, for $2 \leq i \leq n$, let $d_i$ be the smallest number of sets in a chain connecting the other $(n-1)$ pairs of opposite faces in $S$. We have shown that $d_i \gtrsim e^{k-m}$ for each $i$, so Proposition \ref{cubes} gives
$$e^{n(k-m)} \gtrsim \# \mathcal{U} \gtrsim d_1 \cdot e^{(n-1)(k-m)}.$$
Thus, there is a $k$-ball chain of length $\lesssim e^{k-m}$ joining $C_0$ and $C_1$; in particular, such a chain joins $B_x$ and $B_y$. This completes the proof of Lemma \ref{iterlemma}.

\section{Proof of Theorem \ref{mainthm}} \label{quant}

Let $(Z,d)$ be a compact metric space satisfying the assumptions in Theorem \ref{mainthm}. The strongly quasi-M\"obius action $\Gamma \acts Z$ equips $Z$ with a conformal elevator by Lemma \ref{nearfar} (see the remarks following the definition of a conformal elevator). The Ahlfors $n/\e$-regularity of $Z$ immediately implies that every $\delta$-separated set in $Z$ has size at most $C\delta^{-n/\e}$ for some uniform constant $C$. Finally, the discrete length property we impose on $Z$ is precisely the lower bound on discrete paths between points that appears in condition (iv) of Proposition \ref{mainprop}. Thus, $Z$ satisfies all four de-snowflaking conditions, so there is a metric $d_{new}$ on $Z$ for which
$$d(x,y)^{1/\e} \lesssim d_{new}(x,y) \lesssim d(x,y)^{1/\e}.$$

It is an easy exercise to see that the Ahlfors $n/\e$-regularity of $(Z,d)$ translates into Ahlfors $n$-regularity of $(Z,d_{new})$. Of course, $(Z,d_{new})$ remains homeomorphic to $\Sp^n$. More importantly, the action $\Gamma \acts Z$ remains strongly quasi-M\"obius and cocompact on triples with respect to $d_{new}$. The following theorem, which we discussed in Section \ref{intro}, is therefore relevant.

\begin{theorem}[\cite{BK02}, Theorem 1.1] \label{BK2}
Let $n \in \N$, and let $Z$ be a compact, Ahlfors $n$-regular metric space of topological dimension $n$. Suppose that $\Gamma \acts Z$ is a uniformly $\eta$-quasi-M\"obius action on $Z$ that is cocompact on triples. Then $Z$ is $\tilde{\eta}$-quasi-M\"obius equivalent to the sphere $\Sp^n$, where $\tilde{\eta}(t) = C\eta(Ct)$ for some constant $C$.
\end{theorem}

\begin{proof}
The conclusion we state in this theorem is slightly different from that stated in \cite{BK02}. The authors conclude that the action $\Gamma \acts Z$ is quasisymmetrically conjugate to a M\"obius action on $\Sp^n$, but the above statement is implicit on the way to this conclusion. 

We must point out, though, that the authors do not explicitly state the quantitative relationship between $\tilde{\eta}$ and $\eta$. However, the control on $\tilde{\eta}$ that we give here comes from their proof: first establish, as they do, that $Z$ and $\Sp^n$ have bi-Lipschitz equivalent weak-tangents; a quantitative version of \cite[Lemma 2.1]{BK02} gives a quantitative version of \cite[Lemma 5.3]{BK02}, which guarantees that the compactification of a weak tangent of $Z$ is $\eta_1$-quasi-M\"obius equivalent to $Z$, where $\eta_1(t)=C_1\eta(C_1 t)$; the compactification of a weak tangent of $\Sp^n$ is again $\Sp^n$; and the bi-Lipschitz equivalence between weak tangents translates into a strongly quasi-M\"obius equivalence between the compactifications of weak tangents. Putting these facts together gives the desired function $\tilde{\eta}$.
\end{proof}

In our situation, $\Gamma$ acts on $(Z,d_{new})$ by strongly quasi-M\"obius maps, so the distortion function $\tilde{\eta}$ that we obtain from Theorem \ref{BK2} is also linear. Hence, $(Z,d_{new})$ is strongly quasi-M\"obius equivalent to $\Sp^n$. As any strongly quasi-M\"obius homeomorphism between compact sets is necessarily bi-Lipschitz (cf.\ Remark \ref{sqmisbilip}), we find that $(Z,d_{new})$ and $\Sp^n$ are bi-Lipschitz equivalent. Let $\tilde{f} \colon Z \rightarrow \Sp^n$ be a map giving this equivalence, so that
$$d_{new}(x,y) \lesssim d_{\Sp^n}(\tilde{f}(x), \tilde{f}(y)) \lesssim d_{new}(x,y)$$
for all $x,y \in Z$. This completes the proof in the case that $n=1$.

Suppose now that $n \geq 2$. Of course, the map $\tilde{f}$ that we have chosen need not conjugate the action $\Gamma \acts Z$ to a M\"obius action on $\Sp^n$. To correct this, we use a classical theorem of Tukia.

\begin{theorem}[Tukia {\cite[Theorem G]{Tuk86}}] \label{Tukia}
Let $\Gamma$ be a group that acts on $\Sp^n$, $n\geq 2$, by $\eta$-quasi-M\"obius homeomorphisms and is cocompact on triples. Then there is an $\tilde{\eta}$-quasi-M\"obius map $\psi \colon \Sp^n \rightarrow \Sp^n$ for which $\psi \Gamma \psi^{-1}$ is a M\"obius action; here, $\tilde{\eta}(t) = C\eta(t)$ for some constant $C$.
\end{theorem}

\begin{proof}
Again, Tukia's stated result does not include the quantitative relationship between $\tilde{\eta}$ and $\eta$ that we give here. His proof, however, constructs $\psi$ as a limit of maps whose cross-ratio distortion we can keep track of. More specifically, he finds a sequence $g_i \in \Gamma$, corresponding scaling factors $\lambda_i >0$, and a linear map $\alpha \in \GL_n(\R)$ for which 
$$f_i(x) = \hat{\alpha}(\lambda_i \cdot g_i(x))$$
converges to the desired map, $\psi$. Here, $\hat{\alpha}$ is the bi-Lipschitz homeomorphism of $\Sp^n$ obtained from $\alpha$ by conjugation by stereographic projection. Consequently,
$$\begin{aligned}
&\left[ f_i(x_1),f_i(x_2),f_i(x_3),f_i(x_4)\right] \\
&\leq \lVert \hat{\alpha} \rVert^4 [\lambda_ig_i(x_1), \lambda_ig_i(x_2), \lambda_ig_i(x_3),\lambda_ig_i(x_4)] \leq \lVert \hat{\alpha} \rVert^4 \eta([x_1,x_2,x_3,x_4]),
\end{aligned}$$
as scaling by $\lambda_i$ does not change the cross-ratio. We use $\lVert \hat{\alpha} \rVert$ to denote the bi-Lipschitz constant of $\hat{\alpha}$.

Thus, each $f_i$ is $\tilde{\eta}$-quasi-M\"obius with $\tilde{\eta}(t) = \lVert \hat{\alpha} \rVert^4 \eta(t)$, and so the limit function $\psi$ is also $\tilde{\eta}$-quasi-M\"obius.
\end{proof}

Applying this theorem to the strongly quasi-M\"obius action $\tilde{f} \Gamma \tilde{f}^{-1}$ on $\Sp^n$, we obtain a strongly quasi-M\"obius $\psi$, which is therefore also bi-Lipschitz, such that
$$(\psi \circ \tilde{f}) \Gamma (\tilde{f}^{-1} \circ \psi^{-1})$$
is a group of M\"obius transformations on $\Sp^n$. Setting $f= \psi \circ \tilde{f}$ yields the desired $f$.

\begin{remark} \label{n=1conj}
It is not clear whether the stronger conclusion (bi-Lipschitz conjugacy to a M\"obius group) should hold in the case $n=1$. Tukia's theorem has analogs in this setting; see, for example, \cite{Hin90} and \cite{Mar06}, which give us quasisymmetric conjugacy to a M\"obius group. The problem is in choosing the ``correct" conjugacy. Note that there are pairs of cocompact M\"obius groups acting on $\Sp^1$ that are quasisymmetrically conjugate but whose conjugating homeomorphism has non-zero derivative \ti{nowhere}. See \cite{Iva96} for more information about the delicacy of such questions.
\end{remark}

\section{Entropy Rigidity in Coarse Geometry} \label{rcg}

We now turn our attention to Theorem \ref{rigidity}, which is a rigidity result in the setting of Gromov hyperbolic geometry. We refer primarily to \cite{BS07} and \cite{GH90} for background on hyperbolic metric spaces. 

Let $(X,d)$ be a metric space. We say that $X$ is \ti{proper} if all closed balls $\overline{B}(x,r)$ are compact and that $X$ is \ti{geodesic} if any two points can be connected by an isometric image of an interval in $\R$.

Given two metric spaces $(X,d_X)$ and $(Y,d_Y)$, a map $f \colon X \rightarrow Y$ is called a \ti{quasi-isometric embedding} if there are constants $\lambda \geq 1$ and $k \geq 0$ such that
$$\frac{1}{\lambda} d_X(x,x') - k \leq d_Y(f(x),f(x')) \leq \lambda d_X(x,x') + k$$
for all $x,x' \in X$. If, in addition, each point $y \in Y$ lies in the $k$-neighborhood of the image $f(X)$, then we say that $f$ is a \ti{quasi-isometry}. A \ti{rough isometric embedding} or a \ti{rough isometry} is defined in the same way by requiring that $\lambda =1$. For the most part, we will be concerned with rough isometries. When it is necessary to specify the additive constant $k$, we will use the term \ti{$k$-rough isometry}. 

For any three points $x,y,p \in X$, let
$$(x,y)_p = \tfrac{1}{2} \lp d_X(x,p) + d_X(y,p) - d_X(x,y) \rp. $$
This is the \textit{Gromov product} of $x$ and $y$ based at $p$.

\begin{definition}
A metric space $X$ is \textit{$\delta$-hyperbolic} if there is a base-point $p \in X$ so that
\begin{equation} \label{hypeq}
(x,y)_p \geq \min \{(x,z)_p, (y,z)_p \} - \delta
\end{equation}
for every $x,y,z \in X$. We say that $X$ is a \textit{(Gromov) hyperbolic metric space} if it is $\delta$-hyperbolic for some $\delta \geq 0$.
\end{definition}
We will refer to the inequality in \eqref{hypeq} as the \textit{$\delta$-inequality}. Although this definition may seem slightly esoteric, it has a concrete geometric meaning as a ``thinness" condition on triangles. More precisely, if $X$ is a $\delta$-hyperbolic geodesic metric space, then for every geodesic triangle in $X$, each side is contained in the $\delta'$-neighborhood of the union of the other two sides, where $\delta'$ is a constant multiple of $\delta$ (cf.\ \cite[Proposition 2.1.3]{BS07}).

Iterating the $\delta$-inequality, one can obtain a corresponding condition on finite chains of points in $X$. Namely, if $x_0,x_1,\ldots,x_n \in X$, then
$$(x_0,x_n)_p \geq \min_{1\leq i \leq n} (x_i,x_{i-1})_p - \frac{\delta}{\log 2} \log n - c ,$$
where $c$ is a uniform constant depending only on $\delta$ \cite[Chapter 2, Lemma 14(i)]{GH90}. Notice that the smaller we can take $\delta$, the more negatively curved $X$ is. This leads to the following definition, given by M. Bonk and T. Foertsch in \cite{BF06}.

\begin{definition} \label{ACU}
For $\kappa \in [-\infty, 0)$, we say that $X$ has an asymptotic upper curvature bound $\kappa$ if there is $p \in X$ and a constant $c \geq 0$ so that
$$(x_0,x_n)_p \geq \min_{1\leq i \leq n} (x_i,x_{i-1})_p - \frac{1}{\sqrt{-\kappa}} \log n - c$$
for all chains $x_0,\ldots,x_n$ in $X$.
\end{definition}
Here, we use the convention that $1/\sqrt{\infty} = 0$. If $X$ has an asymptotic upper curvature bound $\kappa < 0$, then we say that $X$ is an $\AC_u(\kappa)$-space. By our discussion in the previous paragraph, every hyperbolic metric space is an $\AC_u(\kappa)$-space for some $\kappa <0$. And conversely, the definitions immediately imply that every $\AC_u(\kappa)$-space is Gromov hyperbolic.

Allowing the additive constant $c$ in the definition of asymptotic upper curvature is what makes this notion \textit{asymptotic}. A collection of uniformly bounded configurations in $X$ will not affect the asymptotic curvature bounds, as one could simply make $c$ larger. It makes sense, then, that the best way to study these curvature bounds is to pass to the boundary at infinity, which we now recall.

\subsection{The hyperbolic boundary}

To begin, we say that a sequence $\{x_n\}$ in $X$ \textit{converges at infinity} if 
$$(x_n,x_m)_p \rightarrow \infty \hspace{0.5cm} \text{as} \hspace{0.5cm} n,m\rightarrow \infty.$$ 
It is immediate to see that this property is independent of $p$. We consider two such sequences $\{x_n\}$ and $\{y_n\}$ to be equivalent if 
$$\lim_{n\rightarrow \infty} (x_n,y_n)_p = \infty,$$
and in this case, we write $\{x_n\} \sim \{y_n\}$. This is an equivalence relation on the set of sequences converging at infinity, and we let $\bdry X$ denote the set of equivalence classes. Observe that if a sequence converges at infinity, then any subsequence also converges at infinity and, moreover, is equivalent to the original sequence.

The Gromov product on $X$ extends to $\bdry X$ by
$$(\xi,\eta)_p = \inf \liminf_{n \rightarrow \infty} (x_n,y_n)_p$$
where the infimum is taken over all $\{x_n\}$ and $\{y_n\}$ in the equivalence classes $\xi$ and $\eta$, respectively. Although taking this infimum is necessary in general, the following lemma shows that it is not too restrictive.

\begin{lemma}[\cite{BS07}, Lemma 2.2.2] \label{2.2.2}
Let $X$ be $\delta$-hyperbolic with base-point $p$, and let $\xi,\eta,\zeta \in \bdry X$.
\begin{enumerate}
\item[\textup{(i)}] If $\{x_n\}$ represents $\xi$ and $\{y_n\}$ represents $\eta$, then 
$$(\xi,\eta)_p \leq \liminf_{n\rightarrow \infty}(x_n,y_n)_p \leq \limsup_{n\rightarrow \infty}(x_n,y_n)_p \leq (\xi,\eta)_p + 2\delta.$$
\item[\textup{(ii)}] The $\delta$-inequality $(\xi,\eta)_p \geq \min \{(\xi,\zeta)_p,(\eta,\zeta)_p \} -\delta$ is satisfied.
\end{enumerate}
\end{lemma}

When $X$ is a $\CAT(-1)$-space, Bourdon \cite{Bour96} has shown that
$$\rho(\xi,\eta) = e^{-(\xi,\eta)_p}$$
is a metric on $\bdry X$ and thus gives the boundary a canonical metric. In the more general Gromov hyperbolic setting, however, this function may fail the triangle inequality. In its place, we have
\begin{equation} \label{Keq}
\rho(\xi,\eta) \leq K \max \{ \rho(\xi,\zeta),\rho(\zeta,\eta) \}
\end{equation}
for any $\xi,\eta,\zeta \in \bdry X$, which follows immediately from part (ii) in the preceding lemma. Note that $K=e^{\delta}$ if $X$ is $\delta$-hyperbolic. A general procedure then produces, for $\e$ small enough (depending only on $\delta$), a metric $d_{\e}$ on $\bdry X$ satisfying
$$\tfrac{1}{4} e^{-\e(\xi,\eta)_p} \leq d_{\e}(\xi,\eta) \leq e^{-\e(\xi,\eta)_p}.$$ See \cite[Section 2.2]{BS07}, especially Lemma 2.2.5, for details. This motivates the following definition.

\begin{definition}
A metric $d$ on $\bdry X$ is called a \textit{visual metric of parameter $\e$} if there is a base-point $p \in X$ so that
$$e^{-\e(\xi,\eta)_p} \lesssim d(\xi,\eta) \lesssim e^{-\e(\xi,\eta)_p}$$
for all $\xi,\eta \in \bdry X$. We say that $d$ is \textit{visual} if it is visual with respect to some $\e>0$.
\end{definition}
The dependence on $p$ is not important here; if $d$ is visual with respect to $p$, then it will be visual with respect to any other base-point, with the same parameter $\e$. Observe that if $\bdry X$ admits a visual metric of parameter $\e$, then it admits metrics of all parameters smaller than $\e$. Thus, if we set 
$$\e_0 = \e_0(X) = \sup \{ \e : \text{there is a visual metric on } \bdry X \text{ of parameter } \e \},$$
then each $\e \in (0,\e_0)$ has an associated visual metric. We call this interval the \textit{visual interval}. One should keep in mind the heuristic that the more negatively curved $X$ is, the larger $\e_0$ will be.

The relationship between curvature in $X$ and the length of this visual interval is more explicit in terms of asymptotic upper curvature bounds. Actually, we first need an additional assumption on $X$ to guarantee that its boundary accurately reflects its geometry at large scales. 

\begin{definition} \label{visualdef}
We say that $X$ is \textit{visual} if there is a constant $k$ and a base-point $p \in X$ such that for every $x \in X$ there is a $k$-rough isometric embedding $\gamma \colon [0,\infty) \rightarrow X$ with $\gamma(0)=p$ and $x$ in the image of $\gamma$.
\end{definition}
We will refer to the image of such $\gamma$ as a $k$-rough geodesic ray, starting at $p$. For visual metric spaces, the $\AC_u(\kappa)$ condition can be transferred to the boundary.

\begin{prop} [\cite{BF06}, Lemma 4.1] \label{chains}
Let $X$ be a visual, hyperbolic metric space and assume that there are constants $a$ and $c$ with
$$(\xi_0,\xi_n)_p \geq \min_{1\leq i \leq n} (\xi_i,\xi_{i-1})_p - a\log n - c$$
for all chains $\xi_0,\ldots,\xi_n$ in $\bdry X$. Then there is a constant $c'$ for which
$$(x_0,x_n)_p \geq \min_{1\leq i \leq n} (x_i,x_{i-1})_p - a\log n - c'$$
for all chains $x_0,\ldots,x_n$ in $X$.
Conversely, if the inequality with chains in $X$ holds for some $c'$, then there is a constant $c$ for which the inequality with boundary chains holds.
\end{prop}

This condition on boundary chains gives more precise control on the type of inequality for $\rho$ in \eqref{Keq}. Indeed, we now have
$$\rho(\xi_0,\xi_n) \leq Cn^a \max_{1\leq i \leq n} \rho(\xi_i,\xi_{i-1})$$
for any chain $\xi_0,\ldots,\xi_n$. Arguments similar to those in \cite[Lemma 2.2.5]{BS07} allow one to build visual metrics on $\bdry X$, but this time with more control on the optimal value of $\e_0$. In the end, the authors obtain the following.

\begin{prop}[\cite{BF06}, Theorem 1.5] \label{interval}
Let $X$ be a visual, hyperbolic metric space. If $X$ is $\AC_u(\kappa)$, then for each $0 <\e < \sqrt{-\kappa}$ there is a visual metric on $\bdry X$ with parameter $\e$. Conversely, if there is a visual metric on $\bdry X$ with parameter $\e$, then $X$ is an $\AC_u(-\e^2)$-space.
\end{prop}

Together with other results in \cite{BF06}, this fact suggests that the correct analog of $\CAT(-1)$ in the coarse setting is $\AC_u(-1)$. In the case where $X$ is $\CAT(-1)$, the canonical metric on $\bdry X$ is associated to the parameter $\e_0=1$; in particular, there \textit{are} visual metrics of parameter 1. Unfortunately, this may not happen for more general $\AC_u(-1)$-spaces, even though we know that visual metrics exist for all parameters $0<\e<1$.

\subsection{Geometric actions on hyperbolic metric spaces}

Let $X$ be a proper, geodesic, hyperbolic metric space. These basic assumptions guarantee two important ``accessibility" properties for points in $\bdry X$. First, for any base-point $p \in X$ and each $z \in \bdry X$, there is an isometric embedding $\gamma \colon [0,\infty) \rightarrow X$ for which
$$\gamma(0) = p \hspace{0.3cm} \text{ and } \hspace{0.3cm} \{\gamma(t_n)\} \text{ represents } z$$ whenever $t_n \rightarrow \infty$. We refer to images of such embeddings as geodesic rays and denote them by $[p,z)$. 

Similarly, for any two distinct points $z,z' \in \bdry X$ there is an isometry $\gamma \colon \R \rightarrow X$ for which 
$$\{\gamma(-t_n)\} \text{ represents } z \hspace{0.3cm} \text{ and } \hspace{0.3cm} \{\gamma(t_n)\}\text{ represents } z'$$
whenever $t_n \rightarrow \infty$. Naturally, we will denote such geodesic lines by $(z,z')$. The hyperbolicity of $X$ then guarantees that there is a uniform constant $C$ for which
\begin{equation} \label{GromovLine}
\lab (z,z')_p - \dist(p,(z,z')) \rab \leq C
\end{equation}
whenever $z,z' \in X \cup \bdry X$ are distinct and $p \in X$.

A subset $Y \subset X$ is called \textit{quasi-convex} if there is a constant $C$ for which every geodesic segment in $X$ with endpoints in $Y$ lies in the $C$-neighborhood of $Y$. We then say that an action $\Gamma \acts X$ is \textit{quasi-convex geometric} if the action is
\begin{enumerate}
\item[(i)] isometric (each $g \in \Gamma$ acts as an isometry);
\item[(ii)] properly discontinuous (the set $\{g \in \Gamma : g(K) \cap K \neq \emptyset \}$ is finite for every compact set $K \subset X$);
\item[(iii)] quasi-convex cocompact (there is a non-empty, $\Gamma$-invariant, quasi-convex set $Y \subset X$ and a compact set $K \subset Y$ for which $Y = \bigcup_{g \in \Gamma} g(K)$).
\end{enumerate}

Let us fix such a group action $\Gamma \acts X$ and a corresponding quasi-convex set $Y$. As $Y$ is $\Gamma$-invariant, the action $\Gamma \acts Y$ is isometric, properly discontinuous, and cocompact. Recall that such actions are said to be \ti{geometric}.

For $p \in X$ fixed, the limit set $\Lambda(\Gamma)$ is the collection of points $z \in \bdry X$ that can be represented by a sequence $\{x_n\} \subset \Gamma p$. Of course, this is independent of our choice of $p$. It is not difficult to see that the orbit $\Gamma p$ and the set $Y$ are within finite Hausdorff distance from each other, so $\Lambda(\Gamma)$ coincides with $\bdry Y$, viewed as a subset of $\bdry X$. In particular, $\Lambda(\Gamma)$ is compact.

In fact, it will be convenient simply to replace $Y$ with $\Gamma p$. We lose no generality in doing this, as quasi-convexity of $Y$ implies quasi-convexity of $\Gamma p$. Thus, we take $Y=\Gamma p$ from now on.

Recall from earlier that the entropy of this action $\Gamma \acts X$ is
\begin{equation} \label{exp}
e(\Gamma) = \limsup_{R \rightarrow \infty} \frac{\log(N(R))}{R}
\end{equation}
where $N(R) = \# \{ \Gamma p \cap B_X(p,R) \}$. Under our assumptions, $e(\Gamma) < \infty$ and we can replace the ``$\limsup$" with ``$\lim$"; in fact,
$$\exp(e(\Gamma)R) \lesssim N(R) \lesssim \exp(e(\Gamma)R)$$
(see \cite[Th\'eor\`eme 7.2]{Coor93}). This quantity $e(\Gamma)$ is the coarse analog of volume entropy for Riemannian manifolds, and it is closely related to the metric regularity on $\Lambda(\Gamma)$.

\begin{theorem}[Coornaert {\cite[Section 7]{Coor93}}] \label{visualreg}
When equipped with a visual metric of parameter $\e>0$, the limit set $\Lambda(\Gamma)$ is Ahlfors regular of dimension $e(\Gamma)/\e$.
\end{theorem}

We now wish to transfer the action $\Gamma \acts X$ to a quasi-M\"obius action on $\Lambda(\Gamma)$. Until we mention otherwise, we equip $\Lambda(\Gamma)$ with a visual metric $d$ of parameter $\e$. The following lemma indicates that the induced action on $\Lambda(\Gamma)$ is strongly quasi-M\"obius.

\begin{lemma} \label{extension}
Let $g \in \Gamma$. Then $g$ extends naturally to an $\eta$-quasi-M\"obius homeomorphism of $(\Lambda(\Gamma),d)$, where $\eta(t) = Ct$, and $C$ depends only on the hyperbolicity constant of $X$ and the multiplicative constant in $d$.
\end{lemma}

This lemma and its proof are well known, though most references deal with the more general case when $g$ is assumed to be only a quasi-isometry. In that setting, $g$ still extends to a quasi-M\"obius homeomorphism of the boundary, but the distortion function $\eta$ might not be linear. One does, however, recover a linear distortion function when $g$ is a rough isometry. The proof of Lemma \ref{extension} follows standard extension arguments (see, for example, \cite[Section 4]{Paul96}) and makes use of the following important fact: for $x_1,x_2,x_3,x_4 \in X$, the cross-difference 
$$\begin{aligned}
(x_1,x_3)_p &+ (x_2,x_4)_p - (x_1,x_4)_p - (x_2,x_3)_p \\
&= \tfrac{1}{2} \lp d_X(x_1,x_4) + d_X(x_2,x_3) - d_X(x_1,x_3) - d_X(x_2,x_4) \rp,
\end{aligned}$$
is independent of the chosen base-point $p \in X$. In particular, isometries preserve cross-differences, and so their extensions preserve metric cross-ratios, up to a multiplicative constant.

Abusing terminology, we will continue to let $g$ denote the extension of $g \in \Gamma$ to $\Lambda(\Gamma)$. It is clear that composition is preserved in the extension, so we indeed obtain a strongly quasi-M\"obius group action $\Gamma \acts \Lambda(\Gamma)$. The next lemma, well known in this subject, shows that the cocompactness and proper discontinuity of $\Gamma \acts Y$ extends to cocompactness and proper discontinuity on triples for $\Gamma \acts \Lambda(\Gamma)$. See, for example, \cite[Sections 8.2.K--8.2.Q]{Grom87} for further discussion.

\begin{lemma} \label{triples}
If $\Lambda(\Gamma)$ has at least three points, then the induced action $\Gamma \acts \Lambda(\Gamma)$ is
\begin{enumerate}
\item[\textup{(i)}] cocompact on triples,
\item[\textup{(ii)}] properly discontinuous on triples: for each triple $z_1,z_2,z_3 \in \Lambda(\Gamma)$ of distinct points, for every $\tau > 0$ there are only finitely many $g \in \Gamma$ for which $gz_1$, $gz_2$, and $gz_3$ are $\tau$-separated.
\end{enumerate}
\end{lemma}

Before setting out to prove Theorem \ref{rigidity}, it is necessary to explain what it means for a rough isometry $\Phi \colon S \rightarrow Y$ to be ``roughly equivariant" with respect to a geometric action of $\Gamma$ on $S$. Of course, we will be interested in the case when $S=\Hy^{n+1}$.

\begin{definition}
A map $\Phi \colon S \rightarrow Y$ is \ti{roughly equivariant} with respect to the actions $\Gamma \acts S$ and $\Gamma \acts Y$ if there is a constant $C$ for which
$$d_X \lp \Phi(gx),g\Phi(x) \rp \leq C$$
for each $x \in S$ and $g \in \Gamma$.
\end{definition}

We will shortly need the fact that $\bdry \Hy^{n+1}$ can be identified with $\Sp^n$. Under this identification, the chordal metric on $\Sp^n$ is a visual metric of parameter 1, cf.\ \cite[Section 2.4.3]{BS07}.

\subsection{Proof of Theorem \ref{rigidity}}

Let us return now to the set-up in Theorem \ref{rigidity}. Fix $\Gamma \acts X$ as in the statement of the theorem, and recall that $\Lambda(\Gamma)$ is assumed to be a topological sphere. Let $Y = \Gamma p$, so that $Y$ is quasi-convex and is an $\AC_u(-1)$-space. Using the geometric action $\Gamma \acts Y$, we can verify the following lemma.

\begin{lemma} \label{Yvisual}
There is a uniform constant $C$ such that each $y \in Y$ lies in a $C$-rough geodesic ray in $Y$, starting at $p$. In other words, $Y$ is visual, in the sense of Definition \ref{visualdef}.
\end{lemma}

\begin{proof} 
Fix $x \in Y$. We first want to find a geodesic line $(z,z')$, with $z, z' \in \Lambda(\Gamma)$, that passes close to $x$. To do this, choose two distinct points $w, w' \in \Lambda(\Gamma)$. The quasi-convexity of $Y$ ensures that the geodesic line $(w,w')$ in $X$ lies in the $C_1$-neighborhood of $Y$, for some uniform constant $C_1$. In particular, there is a point $x' \in Y$ for which $\dist(x', (w,w')) \leq C_1$, and there is $g \in \Gamma$ with $gx' = x$. Thus, $\dist(x,(gw,gw')) \leq C_1$. Let $z=gw$ and $z'=gw'$ so that $\dist(x,(z,z')) \leq C_1$.

Consider now the geodesic triangle with sides $[p,z)$, $[p,z')$, and $(z,z')$. The $\delta$-inequality in Lemma \ref{2.2.2}(ii) is valid for points in $X \cup \bdry X$, and this translates into a thinness condition for geodesic triangles, even those with some vertices in $\bdry X$. Consequently, 
$$\dist(x,[p,z)\cup [p,z')) \leq \dist(x,(z,z')) +C_2 \leq C_1+C_2,$$
where $C_2$ is uniform. Thus, we may assume that $\dist(x,[p,z)) \leq C_3$ for a uniform constant $C_3$.

It now suffices to show that $[p,z)$ is in the $C_4$-neighborhood of $Y$, where again $C_4$ is a uniform constant. Indeed, this easily implies that we can find a $C$-rough geodesic ray in $Y$, starting at $p$, and passing through $x$. As $z \in \Lambda(\Gamma)$, there is a sequence $\{x_n\} \subset Y$ that represents $z$, and by quasi-convexity of $Y$, the geodesic segments $[p,x_n]$ lie in the $C_4$-neighborhood of $Y$. The parameterized geodesic ray $[p,z)$ is simply the limit of the parameterized segments $[p,x_n]$ (in the topology of uniform convergence on compact sets), so we immediately see that $[p,z)$ is also in the $C_4$-neighborhood of $Y$.
\end{proof}

As $Y$ is visual and $\AC_u(-1)$, we can apply Proposition \ref{interval} to obtain, for each $0<\e<1$, a visual metric on $\bdry Y$ of parameter $\e$. Recall, though, that $\bdry Y$ coincides with $\Lambda(\Gamma)$. Thus, there are visual metrics on $\Lambda(\Gamma)$ for all parameters $0<\e<1$. By Theorem \ref{visualreg}, these metrics are Ahlfors regular of dimension $e(\Gamma)/\e$. In particular, the Hausdorff dimension of $\Lambda(\Gamma)$ with this metric is $e(\Gamma)/\e$. 

On the other hand, all visual metrics induce the same topology on $\Lambda(\Gamma)$; in our case, this is the topology of the standard $n$-dimensional sphere. Recalling that the topological dimension of a compact metric space always bounds the Hausdorff dimension from below (cf.\ \cite[Theorem 7.2]{HW41}), we obtain
$$\frac{e(\Gamma)}{\e} = \dim_H(\Lambda(\Gamma),d_{\e}) \geq \dim_{\topo}(\Lambda(\Gamma),d_{\e}) = n$$
for all $0<\e<1$. This gives $e(\Gamma) \geq n$, which is the first part of the theorem.

It remains to prove the rigidity statement in Theorem \ref{rigidity}, and this task will occupy us for the remainder of the section. The ``if" part of the statement follows easily from standard facts about hyperbolic metric spaces. Namely, if $\Phi \colon \Hy^{n+1} \rightarrow Y$ is a rough isometry, then the fact that $\bdry \Hy^{n+1} = \Sp^n$ admits a visual metric of parameter 1 implies that $\bdry Y = \Lambda(\Gamma)$ does as well. Equipped with these metrics, we can extend $\Phi$ to a bi-Lipschitz map of the boundaries:
$$\Phi \colon \Sp^n \rightarrow \Lambda(\Gamma).$$
In particular,
$$e(\Gamma) = \dim_H \Lambda(\Gamma) = n.$$

Let us now address the converse statement. Thus, we assume that $e(\Gamma)=n$ and wish to construct the desired action $\Gamma \acts \Hy^{n+1}$ and map $\Phi \colon \Hy^{n+1} \rightarrow Y$.

Fix a visual parameter $0<\e<1$ for $Z=\Lambda(\Gamma)$, which we also view as $\bdry Y$, and let $d$ denote a corresponding visual metric. Then $(Z,d)$ is Ahlfors $n/\e$-regular and Lemma \ref{extension} implies that there is a strongly quasi-M\"obius action $\Gamma \acts Z$. Moreover, Lemma \ref{triples}(i) guarantees that this action is cocompact on triples. We claim that the discrete length condition appearing in Theorem \ref{mainthm} follows from the $\AC_u(-1)$ assumption on $Y$. Indeed, if 
$$u=z_0,z_1,\ldots,z_l=v$$
is a discrete $\delta_0$-path between $u$ and $v$ in $Z$, then Proposition \ref{chains} gives
$$(u,v)_p \geq \min_{1\leq i \leq l} (z_i,z_{i-1})_p - \log l - c$$
for some uniform constant $c$. Translating this to the metric, we obtain
$$d(u,v) \lesssim l^{\e} \cdot \max_{1\leq i \leq l} d(z_i,z_{i-1}) \lesssim \delta_0 l^{\e},$$
and rearranging gives $l \gtrsim \lp d(u,v)/\delta_0 \rp^{1/\e}$. The conditions in Theorem \ref{mainthm} are therefore satisfied, and so we obtain a metric $d_{new}$ for which $d$ and $d_{new}^{\e}$ are bi-Lipschitz equivalent. In particular, $d_{new}$ is a visual metric on $Z$ of parameter $1$. We also obtain a bi-Lipschitz map $f \colon \Sp^n \rightarrow (Z,d_{new})$ for which $f^{-1} \Gamma f$ is a M\"obius action on the sphere. Observe that this action is cocompact on triples.

Furthermore, as $\Gamma \acts X$ is properly discontinuous, the induced boundary action $\Gamma \acts Z$ will be properly discontinuous on triples by Lemma \ref{triples}(ii). This property is preserved under conjugation by homeomorphisms, so the M\"obius action $f^{-1} \Gamma f$ will also be properly discontinuous on triples.

By the correspondence between M\"obius transformations on $\Sp^n$ and isometries of $\Hy^{n+1}$, for each $g \in \Gamma$, there is a unique isometry of $\Hy^{n+1}$ that induces $f^{-1}gf$ on the boundary. This gives us a geometric action $\Gamma \acts \Hy^{n+1}$. Indeed, the cocompactness and proper discontinuity on triples for $f^{-1} \Gamma f$ translate into cocompactness and proper discontinuity for $\Gamma \acts \Hy^{n+1}$.

To construct $\Phi$, we use standard arguments about extending bi-Lipschitz maps between boundaries of hyperbolic metric spaces to rough isometries of the hyperbolic spaces themselves. Actually, our argument will mimic the proof of Theorem 7.1.2 in \cite{BS07}. Important to this construction is again the fact that the chordal metric on $\Sp^n$ is a visual metric of parameter $1$ under the identification of $\Sp^n$ with $\bdry \Hy^{n+1}$. Thus, $f \colon \Sp^n \rightarrow (Z,d_{new})$ is a bi-Lipschitz homeomorphism between two spaces whose metrics are of the form $e^{-(u,v)}$, up to multiplicative constants.

Fix $x \in \Hy^{n+1}$; for concreteness we will use the unit ball model of $\Hy^{n+1}$. Then there is a geodesic ray $[0,z)$ in $\Hy^{n+1}$, ending at some $z \in \Sp^n$, with $x \in [0,z)$. Let $\gamma \colon [0,\infty) \rightarrow \Hy^{n+1}$ be the unit speed parameterization of this ray, and let $t = d_{\Hy^{n+1}}(0,x)$ so that $\gamma(t) = x$. Now, as $f(z)$ is in $\Lambda(\Gamma)$, we also know that there is a geodesic ray $[p,f(z))$ in $X$ that lies in the $C_1$-neighborhood of $Y$ (cf.\ the proof of Lemma \ref{Yvisual}). Let $\tilde{\gamma} \colon [0,\infty) \rightarrow X$ be the geodesic parameterization of this ray. We then define $\Phi(x)$ to be a point in $Y$ that is of distance at most $C_1$ from $\tilde{\gamma}(t)$. Of course, this definition depends on the choice of a ray $[p,f(z))$ and on the choice of a point in $Y$. Making different choices, however, yields points that are within distance $C_2$ of each other, where $C_2$ is uniform.

The map $\Phi \colon \Hy^{n+1} \rightarrow Y$ thus defined induces, almost by definition, the homeomorphism $f$ between $\Sp^n$ and $Z$. Moreover, we claim that $\Phi$ is a rough isometry. To prove the desired bounds on $d_X(\Phi(x),\Phi(y))$, it suffices to show that
\begin{equation} \label{Phi}
\lab (x,y)_0 - (\Phi(x),\Phi(y))_p \rab \leq C_3
\end{equation}
for a constant $C_3$ independent of $x,y \in \Hy^{n+1}$. Indeed, the definition of $\Phi$ guarantees that 
\begin{equation} \label{Phi2}
\lab d_{\Hy^{n+1}}(0,x) - d_X(p,\Phi(x)) \rab \hspace{0.05cm}, \hspace{0.05cm} \lab d_{\Hy^{n+1}}(0,y) - d_X(p,\Phi(y)) \rab \leq C_4
\end{equation}
for a uniform constant $C_4$.

Fix $x,y \in \Hy^{n+1}$ and let $u,v \in \Sp^n$ be boundary points for which $x \in [0,u)$ and $y \in [0,v)$. The metric hyperbolicity of $\Hy^{n+1}$ implies that
$$\lab (x,y)_0 - \min \{d_{\Hy^{n+1}}(0,x), d_{\Hy^{n+1}}(0,y), (u,v)_0 \} \rab$$
is uniformly bounded (cf.\ \cite[Lemma 7.1.3]{BS07}). Similarly, the hyperbolicity of $X$ ensures that
$$\lab (\Phi(x),\Phi(y))_p - \min\{ d_X(p,\Phi(x)), d_X(p, \Phi(y)), (f(u),f(v))_p \} \rab $$
is uniformly bounded. Using \eqref{Phi2} again, it is clear that \eqref{Phi} would follow from the uniform boundedness of
$$\lab (u,v)_0 - (f(u),f(v))_p \rab.$$
This, however, is an immediate consequence of the fact that $f$ is bi-Lipschitz with respect to visual metrics of parameter 1; observe that the bound will depend on the bi-Lipschitz constant of $f$. Thus, \eqref{Phi} holds, and so
$$d_{\Hy^{n+1}}(x,y) - C_5 \leq d_X(\Phi(x),\Phi(y)) \leq d_{\Hy^{n+1}}(x,y) + C_5,$$
where $C_5$ is uniform. Note also that the definition of $\Phi$, along with the facts that $f$ is surjective and $Y$ is visual, imply that each point in $Y$ is of distance at most $C_6$ from $\Phi(\Hy^{n+1})$. Thus, $\Phi$ is a rough isometry.

Finally, we must show that $\Phi$ is roughly equivariant. To this end, let $g \in \Gamma$ and consider the rough isometry $g^{-1} \Phi g \colon \Hy^{n+1} \rightarrow Y$. Observe that it extends to the map 
$$g^{-1} f g \colon \Sp^n \rightarrow \Lambda(\Gamma).$$
As $f$ is equivariant with respect to the boundary actions, we know that $g^{-1} f g = f$. Thus, $g^{-1} \Phi g$ and $\Phi$ are rough isometries whose boundary extensions coincide. This implies that there is a uniform constant $C$ for which
$$d_X(g^{-1}\Phi(gx), \Phi(x)) \leq C$$
whenever $x \in \Hy^{n+1}$ (cf.\ \cite[Proposition 9.1]{BKM09}). Hence,
$$d_X(\Phi(gx), g\Phi(x)) \leq C$$
for each $x \in \Hy^{n+1}$ and $g \in \Gamma$. This completes the proof of Theorem \ref{rigidity}.

\begin{remark} \label{n=1r}
Many of the arguments we have used are valid in the case $n=1$ as well. In particular, we can conclude that $e(\Gamma) \geq 1$ and that if $\Gamma p$ is roughly isometric to $\Hy^2$, then $e(\Gamma)=1$. The notable exception is the argument that allows us to conjugate the action $\Gamma \acts \Lambda(\Gamma)$ to a M\"obius action on $\Sp^1$ \ti{by a bi-Lipschitz map}. We discussed this issue in Remark \ref{n=1conj}, where we also indicated that the conjugation is possible with a quasisymmetric map. By standard extension arguments similar to those we used above to construct $\Phi$, one can extend the quasisymmetric conjugation map to a \ti{quasi-isometry} between $\Hy^2$ and $\Gamma p$. This quasi-isometry will still be roughly equivariant with respect to the actions of $\Gamma$ on $\Hy^2$ and on $\Gamma p$ by the same arguments we used above. 

Thus, when $n=1$, we can say only that $\Phi$ will be a quasi-isometry, rather than a rough isometry. We leave as an open question whether the stronger conclusion holds.
\end{remark}


\end{document}